\documentclass[11pt]{article}

\usepackage[margin=1in]{geometry}
\usepackage{amsmath,amsfonts,amsthm,amssymb}

\usepackage{cite}
\usepackage{subcaption}
\usepackage{graphicx}
\graphicspath{{figures/}}
\usepackage{color}
\usepackage{url}
\usepackage{soul}
\usepackage{enumitem}

\usepackage{todonotes}

\usepackage{verbatim}

\newcommand{\bbR}{\mathbb{R}}

\newcommand{\Rn}{\mathbb{R}^n}

\newcommand{\Zn}{\mathbb{Z}^n}

\newcommand{\Ldel}{L^{\delta,\beta}}
\newcommand{\udel}{u^{\delta,\beta}}
\newcommand{\cdel}{c^{\delta,\beta}}
\newcommand{\mdel}{m^{\delta,\beta}}

\newtheorem{theorem}{Theorem}

\newtheorem{cor}{Corollary}
\newtheorem{lem}{Lemma}

\newtheorem{prop}{Proposition}

\theoremstyle{remark}
\newtheorem{remark}{Remark}

\title{Regularity of Solutions for the Nonlocal Diffusion Equation on Periodic Distributions}
\author{Ilyas Mustapha, Bacim Alali, and Nathan Albin\\
\
\footnotesize{Department of Mathematics, Kansas State University, Manhattan, KS}}

\begin{document}
\maketitle

\begin{abstract}
    This work addresses  the regularity of solutions for a nonlocal diffusion equation over the space of periodic distributions. The spatial operator for the nonlocal diffusion equation is  given by a nonlocal Laplace operator with a compactly supported integral kernel. We follow a unified approach based on the Fourier multipliers of the nonlocal Laplace operator, which allows 
    the study of regular as well as distributional solutions of the nonlocal diffusion equation, integrable as well as singular kernels, in any spatial dimension. In addition, the results extend beyond operators with singular kernels to  nonlocal super-diffusion operators. We present results on the spatial and temporal regularity of solutions in terms of regularity of the initial data or the diffusion source term. Moreover, solutions of the nonlocal diffusion equation are shown to converge to the solution of the classical diffusion equation for two types of limits: as the spatial nonlocality vanishes or as the singularity of the integral kernel approaches a certain critical  singularity that depends on the spatial dimension. Furthermore, 
    we show that, for the case of integrable kernels, discontinuities in the initial data propagate and persist in the solution of the nonlocal diffusion equation. The magnitude of a jump discontinuity is shown to decay overtime.
\end{abstract}

\noindent \textit{Keywords:} Nonlocal diffusion equations, nonlocal Laplace operators,  nonlocal superdiffusion, Fourier multipliers, spatial regularity, temporal regularity.


\section{Introduction}
In this work, we study  the regularity of solutions to the nonlocal diffusion equation given by 
\begin{eqnarray}
\begin{cases}
    u_t(x,t)=\Ldel u(x,t)+b(x) ,~~x\in T^n,~t>0,\\
    u(x,0) = f(x),
\end{cases}\label{nonlocal_diffusion_eqn}
\end{eqnarray}
over the space of periodic distributions $H^s(T^n)$, with $s\in\mathbb{R}$. Here  
 $T^n$ denotes the periodic torus in $\Rn$ and $\Ldel$ is a nonlocal Laplace operator 
defined by
\begin{eqnarray}
    \Ldel u(x)=\cdel\int_{B_\delta(x)}\frac{u(y)-u(x)}{\|y-x\|^\beta}dy,~~x\in\Rn, \label{nonlocal_operator}
\end{eqnarray}
where $B_\delta(x)$ denotes a ball in $\Rn$, $\delta>0$ is called the horizon or the nonlocality, and the kernel exponent $\beta$ satisfies $\beta<n+2$\cite{du2012analysis,nonlocal_calc_2013}. The scaling constant $c^{\delta,\beta}$ is given by 
\begin{eqnarray*}
    \cdel = \frac{2(n+2-\beta)\Gamma(\frac{n}{2}+1)}{\pi^{\frac{n}{2}}\delta^{n+2-\beta}}.
\end{eqnarray*}
Nonlocal integral operators with compact support of the form \eqref{nonlocal_operator} have their roots in  peridynamics \cite{silling2000reformulation, silling2007peridynamic} and have been introduced in nonlocal vector Calculus \cite{nonlocal_calc_2013}.
These nonlocal operators have been used in different applied settings, see for example \cite{gilboa2009nonlocal,hu2012peridynamic, bobaru2015cracks,bucur2016nonlocal,cheng2015peridynamic}. 
The work in \cite{bobaru2010peridynamic} proposed a nonlocal model for transient heat transfer, which is valid when the body undergoes damage or evolving cracks. There have been many mathematical analysis studies involving nonlocal Laplace operators and peridynamic operators including the works \cite{MengeshaDu,kassmann2019solvability,foss2016differentiability,mengesha2020solvability,du2012analysis,alali2021fourier}. In general, exact solutions are not readily available for nonlocal models, however, different computational techniques and numerical analysis methods have been developed for solving nonlocal equations such as  \cite{d2020numerical,du2016asymptotically,borm2010efficient,du2018numerical,jafarzadeh2020efficient,coclite2020numerical,du2017fast,FossRaduYu,Kun2010,alali2020fourier}. 


 The work in \cite{alali2021fourier} introduces the Fourier multipliers for nonlocal Laplace operators, studies the asymptotic behavior of these multipliers, and then applies the asymptotic analysis in the periodic setting to prove  regularity results for the nonlocal Poisson equation. In this work, we apply the Fourier multipliers approach developed in \cite{alali2021fourier} to study the regularity of solutions to the nonlocal diffusion equation over the space of periodic distributions. 
 The organization of this article and a brief description of the main contributions of this study are as
follows.
\begin{itemize}[leftmargin=*,itemsep=.005cm]
\item A review of the Fourier multipliers analysis for the nonlocal Laplace operator \eqref{nonlocal_operator} is provided in Section~\ref{sec:multipliers}.

    \item In Section~\ref{sec:regularity}, we present the regularity of solutions analysis for the nonlocal diffusion equation with initial data in $H^s(T^n)$, but without a diffusion source.
    \begin{itemize}
        \item Theorem~\ref{spatial_regularity} and Proposition~\ref{prop_du=v} provide the spatial and temporal regularity results, respectively, in any spatial dimension. The temporal regularity for a general periodic distribution in $H^s(T^n)$, with $s\in\mathbb{R}$, is studied in the sense of Gateaux derivative.
        \item In the case when the Fourier coefficients of the initial data $f\in H^s(T^n)$ are summable  
        \[
            \underset{k\in\Zn}{\sum}|\hat{f}_k|<\infty,
        \]
        then the solution of the nonlocal diffusion equation,  considered as a function of the spatial variable $x$, is a regular $L^2(T^n)$ function and    Proposition~\ref{temporal_regularity} of Section~\ref{regular_fns_as_sol} provides the temporal regularity of the solution with respect to the classical derivative.
        \item Theorem~\ref{convergence_NHE_Without_diffusionSource1} and Theorem~\ref{convergence_NHE_Without_diffusionSource2} provide convergence results for the solution of the nonlocal diffusion equation, without a diffusion source, to the solution of the corresponding classical diffusion equation with respect to two different limits: as $\delta\to 0^+$ or as $\beta\to n+2$, respectively.
    \end{itemize}
    \item In Section~\ref{NHE_WithdiffusionSource}, we present the regularity of solutions analysis for the nonlocal diffusion equation, when a diffusion source $b\in H^s(T^n)$, for some $s\in\mathbb{R}$, is present.
    \begin{itemize}
        \item Theorem~\ref{U_2spatial_regularity} and Proposition~\ref{prop_du_2=v} provide the spatial and temporal regularity results, respectively, in any spatial dimension. The temporal regularity for a general periodic distribution in $H^s(T^n)$, with $s\in\mathbb{R}$, is studied in the sense of Gateaux derivative.
        \item In the case when the Fourier coefficients of the source term $b\in H^s(T^n)$ are summable  
        \[
            \underset{k\in\Zn}{\sum}|\hat{b}_k|<\infty,
        \]
        then the solution of the nonlocal diffusion equation,  considered as a function of the spatial variable $x$, is a regular $L^2(T^n)$ function and
        Proposition~\ref{U_2temporal_regularity} of Section~\ref{regular_fns_as_sol2} provides the temporal
        regularity of the solution with respect to the classical derivative.
        \item Theorem~\ref{convergence_NHE_With_diffusionSource1} and Theorem~\ref{convergence_NHE_With_diffusionSource2} provide convergence results for the solution of nonlocal diffusion equation with a non-zero diffusion source to the solution of the corresponding classical diffusion equation with respect to two kinds of limits: as $\delta\to 0^+$ or as $\beta\to n+2$, respectively.
    \end{itemize}
    \item In Section~\ref{propagation_of_discontinuities}, we show that, for the case of integrable kernels, that is when $\beta<n$,   discontinuities in the initial data propagate and persist in the solution of the nonlocal diffusion equation. The  magnitude of a jump discontinuity is shown to decay as time increases.
\end{itemize}

\section{Fourier multipliers}\label{sec:multipliers}
In this section, we give a summary of the Fourier multipliers' results introduced in \cite{alali2021fourier}, which are relevant to the work presented in Section \ref{sec:regularity}. These multipliers are defined through the Fourier transform by
\begin{eqnarray}
    \Ldel u(x)=\frac{1}{(2\pi)^n}\int_{\bbR^n}\mdel \hat{u}(\nu)e^{i\nu\cdot x}d\nu,\label{FourierLdel}
\end{eqnarray}
where $\mdel(\nu)$ is given by 
\begin{eqnarray}
    \mdel(\nu)=\cdel\int_{B_\delta(0)}\frac{\cos(\nu\cdot z)-1}{\|z\|^\beta}dz,\label{multipliers}
\end{eqnarray}
    for $\beta<n+2$.
The following theorem gives the hypergeometric representation of these multipliers.
\begin{theorem}
Let $n\ge 1,~\delta>0$ and $\beta<n+2$. Then the Fourier multipliers can be written as
\begin{eqnarray}
    \mdel(\nu) = -\|\nu\|^2 ~_2F_3\left(1,\frac{n+2-\beta}{2};2,\frac{n+2}{2},\frac{n+4-\beta}{2};-\frac{1}{4}\|\nu\|^2\delta^2\right).
    \label{eq:multiplier-general}
\end{eqnarray}
\end{theorem}
The hypergeometric function $_2F_3$ on the right hand side is well-defined for any $\beta\ne n+4,n+6,\cdots$, hence, using \eqref{eq:multiplier-general}, the definition of the multipliers is extended to the case when $\beta\ge n+2$ with $\beta\ne n+4, n+6,\cdots$. 
Consequently, the operator $\Ldel$ is extended to these larger  values of $\beta$ using the Fourier transform.
In particular, for the case  when $\beta=n+2$, and since $m^{\delta,n+2}(\nu)$ is equal to $-\|\nu\|^2$, the extended operator $\Ldel$ coincides with the classical Laplace operator $\Delta$. For the case, $n+2<\beta<n+4$, the extended operator $\Ldel$ corresponds to a {\it nonlocal super-diffusion operator} \cite{alali2020fourier}.

The representation \eqref{eq:multiplier-general} is used to provide the asymptotic behavior of $\mdel(\nu)$ for large $\|\nu\|$. This is given by the following result \cite{alali2021fourier}.
\begin{theorem}\label{thm:asymptotic_behavior}
Let $n\ge 1$, $\delta>0$ and $\beta\notin\{n+2,n+4,n+6,\dots\}$. Then, as $\|\nu\|\to \infty$,
\begin{eqnarray}
\label{eq:asym_multip}
    \mdel(\nu) \sim \begin{cases}
        -\frac{2n(n+2-\beta)}{\delta^2(n-\beta)}+2\left(\frac{2}{\delta}\right)^{n+2-\beta}\frac{\Gamma\left(\frac{n+4-\beta}{2}\right)\Gamma\left(\frac{n+2}{2}\right)}{(n-\beta)\Gamma\left(\frac{\beta}{2}\right)}\|\nu\|^{\beta-n},~~\text{ if }\beta\ne n,\\[1ex]
        -\frac{2n}{\delta^2}\left(2\log\|\nu\|+\log\left(\frac{\delta^2}{4}\right)+\gamma-\psi\left(\frac{n}{2}\right)\right),~~\text{ if }\beta=n,
    \end{cases}
\end{eqnarray}
where $\gamma$ is Euler's constant and $\psi$ is the digamma function.

\end{theorem}

To simplify the notation, throughout this article we will denote $\mdel$ simply by $m$. However, in places in which there is a need to emphasize the dependence of the multipliers on the parameters $\delta$ and $\beta$, such as when we take limits in those parameters, we will revert to the notation $\mdel$.  

\section{Regularity of solutions for the peridynamic diffusion equation}\label{sec:regularity}
In this section, we focus on the following nonlocal diffusion equation with initial data and no diffusion source 
\begin{eqnarray}
\label{eq:nonlocal_diffusion_nosource}
\begin{cases}
    u_t(x,t)=\Ldel u(x,t) ,~~x\in T^n,~t>0,\\
    u(x,0) = f(x).
\end{cases}\label{nonlocal_diffusion_eqn1}
\end{eqnarray}
In order to study the existence, uniqueness, and regularity of solutions to \eqref{nonlocal_diffusion_eqn1} over the space of periodic distributions, we consider the identification $U(t)=u(\cdot,t)$, with $U:[0,\infty)\rightarrow H^s(T^n)$.

\subsection{Eigenvalues on periodic domains}
Let $\Ldel$ be defined on the periodic torus 
\begin{eqnarray*}
    T^n = \prod_{i=1}^n [0,r_i],~r_i>0, ~i=1,2,\cdots,n.
\end{eqnarray*}
Define 
$$\nu_k=\left(\frac{2\pi k_1}{r_1},\frac{2\pi k_2}{r_2},\dots,\frac{2\pi k_n}{r_n}\right),$$
for any $k\in \Zn$. Let $\phi_k(x)=e^{i\nu_k\cdot x}$. Then, 
\begin{eqnarray}
    \Ldel \phi_k(x) = m^{\delta,\beta}(\nu_k)\phi_k(x),\label{eigenvalues_L}
\end{eqnarray}
which shows that $\phi_k$ is an eigenfunction of $\Ldel$ with eigenvalues $m^{\delta,\beta}(\nu_k)$. To simplify the notation, we will often suppress the dependence of the multipliers on $\delta$ and $\beta$ and use $m(\nu)$ to denote $m^{\delta,\beta}(\nu)$. 

\vspace{10pt}
\noindent Consider the nonlocal diffusion equation defined in \eqref{nonlocal_diffusion_eqn}. For $s\in\mathbb{R}$, let $H^s(T^n)$ be the space of periodic distributions $h$ on $T^n$ such that
\begin{eqnarray*}
\|h\|_{H^s(T^n)}^2:=    \sum_{k\in\Zn}(1+\|k\|^2)^s|\hat{h}_k|^2<\infty.
\end{eqnarray*}
\subsection{Distributional solutions for nonlocal diffusion equation}
Let $f\in H^s(T^n)$ and define $U, V:[0,\infty)\to H^q(T^n)$ for some $q\in\mathbb{R}$, by
\begin{eqnarray}
    U(t)&=&\sum_k\hat{f}_ke^{m(\nu_k)t}e^{i\nu_k\cdot x},\label{U}
    \\V(t)&=&\sum_k\hat{f}_km(\nu_k)e^{m(\nu_k)t}e^{i\nu_k\cdot x}.\label{V}
\end{eqnarray}
Observe that for any $t\ge 0$,  $U(t)$ and $V(t)$ are well-defined periodic distributions,  since $e^{m(\nu_k)t}$ and $m(\nu_k)e^{m(\nu_k)t}$ are both bounded functions in $k$.

\begin{theorem}\label{spatial_regularity}
    Let $n\ge 1,~\delta>0$ and $\beta< n+4$. Let $\epsilon_1$ and $\epsilon_2$ be such that $0<\epsilon_1<\epsilon_2<1$. Assume that $f\in H^s(T^n)$ for some $s\in\bbR$. Then for any fixed $t>0$, $U(t)\in H^p(T^n)$ and $V(t)\in H^{r}(T^n)$, where
    \begin{eqnarray}
	p&=&\begin{cases}
	s, \text{ if }\beta< n,\\
	s+\dfrac{4nt}{\delta^2}(1-\epsilon_1),\text{ if }\beta=n,\\
	\infty, \text{ if } \beta> n, 
	\end{cases}~~
	r=\begin{cases}
	s, \text{ if }\beta< n,\\
	s+\dfrac{4nt}{\delta^2}(1-\epsilon_2),\text{ if }\beta=n,\\
	\infty, \text{ if } \beta> n, 
	\end{cases}\label{p_values}
\end{eqnarray}
with $H^\infty(T^n) := \underset{s \in \bbR}{\bigcap}H^s(T^n)$.
\end{theorem}
\begin{proof}
    We observe that
\begin{eqnarray*}
    \sum_{0\ne k\in \mathbb{Z}^n}(1+\|k\|^2)^p|\hat{U}_k|^2= \sum_{0\ne k\in \mathbb{Z}^n} (1+\|k\|^2)^{p-s}(1+\|k\|^2)^s|\hat{f}_ke^{m(\nu_k)t}|^2.
\end{eqnarray*}
Since $f\in H^s(T^n)$, then the result that $U(t)\in H^p(T^n)$ follows by showing that
$$(1+\|k\|^2)^{p-s}e^{2m(\nu_k)t}$$
is bounded for $k\ne 0$. To see this, we consider three cases. For the case $\beta<n$, then $p-s=0$ and 
$$
	e^{2m(\nu_k)t}(1+\|k\|^2)^{p-s}= e^{2m(\nu_k)t},
$$
which is bounded since $m(\nu_k)\le 0$.

 For the case when $n<\beta<n+4$, let $q\in\bbR$ be arbitrary. Then,  
\begin{eqnarray*}
	(1+\|k\|^2)^{q-s}e^{2m(\nu_k)t}=\frac{(1+\|k\|^2)^{q-s}}{e^{2t|m(\nu_k)|}},
\end{eqnarray*}
which vanishes as $\|k\|\to \infty$, and hence boundedness follows. Thus, $U(t)\in H^q(T^n)$ for all $q$ and therefore,  
$$U(t)\in\underset{q \in \bbR}{\bigcap}H^q(T^n)= H^\infty(T^n).$$

For the case when $\beta=n$, then $p-s=\frac{4nt}{\delta^2}(1-\epsilon_1)$. From Theorem \ref{thm:asymptotic_behavior}, we have
\[
    m(\nu_k)\sim -\frac{4n}{\delta^2}\log\|\nu_k\|,
\]
which implies that 
\begin{eqnarray}
    \lim_{\|\nu_k\|\to\infty}\frac{m(\nu_k)}{-\frac{4n}{\delta^2}\log\|\nu_k\|}=1.
\end{eqnarray}
Thus, for any $\epsilon_1>0$, there exists $N\in\mathbb{N}$ such that
\begin{eqnarray}
    -\frac{4n}{\delta^2}(1+\epsilon_1)\log\|\nu_k\|\le m(\nu_k)\le -\frac{4n}{\delta^2}(1-\epsilon_1)\log\|\nu_k\|,\label{asymptotic_inequality}
\end{eqnarray}
for all $\|\nu_k\|\ge N$. Therefore,
\begin{eqnarray}
    e^{2m(\nu_k)t}&\le& e^{-\frac{8nt}{\delta^2}(1-\epsilon_1)\log\|\nu_k\|}
    =\|\nu_k\|^{-\frac{8nt}{\delta^2}(1-\epsilon_1)}.
    \label{eq:exp_estimate}
\end{eqnarray}
Since there exists $A>0$  such that $A\|k\|\le \|\nu_k\|,$ then
\begin{eqnarray*}
	(1+\|k\|^2)^{p-s}e^{2m(\nu_k)t}&\le&\frac{(1+\|k\|^2)^{\frac{4nt}{\delta^2}(1-\epsilon_1)}}{\|\nu_k\|^{\frac{8nt}{\delta^2}(1-\epsilon_1)}}\le \frac{(1+\|k\|^2)^{{\frac{4nt}{\delta^2}(1-\epsilon_1)}}}{(A\|k\|)^{\frac{8nt}{\delta^2}(1-\epsilon_1)}},
\end{eqnarray*}
which is bounded.

Similarly, to show that $V(t)\in H^{r}(T^n),$ we   show that 
$$(1+\|k\|^2)^{r-s}m(\nu_k)^2e^{2m(\nu_k)t}$$
is bounded for $k\ne 0.$ For the case when $\beta<n$, then $r-s=0$ and there exists a constant $C>0$ such that $|m(\nu_k)|\le C$. Thus,
$$(1+\|k\|^2)^{r-s}m(\nu_k)^2e^{2m(\nu_k)t}=m(\nu_k)^2e^{2m(\nu_k)t}$$
is bounded. For the case when $n<\beta<n+4$, then for an arbitrary $q'\in\bbR$, we have  
\begin{eqnarray*}
	(1+\|k\|^2)^{q'-s}m(\nu_k)^2e^{2m(\nu_k)t}=\frac{m(\nu_k)^2(1+\|k\|^2)^{q'-s}}{e^{2t|m(\nu_k)|}},
\end{eqnarray*}
which vanishes as $\|k\|\to\infty$, and therefore boundedness follows. Thus, $V(t)\in H^{q'}(T^n)$ for all $q'\in\bbR$ and hence,
$$V(t)\in\underset{q' \in \bbR}{\bigcap}H^{q'}(T^n)= H^\infty(T^n).$$
When $\beta=n,$ then $r-s=\frac{4nt}{\delta^2}(1-\epsilon_2)$. From \eqref{asymptotic_inequality},
\begin{equation}
\label{eq:msquared}
    |m(\nu_k)|^2\le \left(\frac{4n}{\delta^2}(1+\epsilon_1)\right)^2(\log\|\nu_k\|)^2.
\end{equation}
In addition, there exists $N_2\in\mathbb{N}$ such that 
\begin{eqnarray}
    \log(\|\nu_k\|)\le\|\nu_k\|^{\frac{4nt}{\delta^2}(\epsilon_2-\epsilon_1)},\label{log-upperbound}
\end{eqnarray}
for all $\|\nu_k\|>N_2$. Moreover, there exists $B>0$  such that $\|\nu_k\|\le B\|k\|$. Hence, by using \eqref{eq:exp_estimate}. \eqref{eq:msquared}, and \eqref{log-upperbound}, we obtain 
\begin{eqnarray*}
    (1+\|k\|^2)^{r-s}|m(\nu_k)|^2e^{2m(\nu_k)t}&\le&\frac{(1+\|k\|^2)^{\frac{4nt}{\delta^2}(1-\epsilon_2)}\left(\frac{4n}{\delta^2}(1+\epsilon_1)\right)^2(B^2\|k\|^2)^{\frac{4nt}{\delta^2}(\epsilon_2-\epsilon_1)}}{(A\|k\|)^{\frac{8nt}{\delta^2}(1-\epsilon_1)}}\\
    &=&C \frac{(1+\|k\|^2)^{\frac{4nt}{\delta^2}(1-\epsilon_2)}}{\|k\|^{\frac{8nt}{\delta^2}(1-\epsilon_2)}},
\end{eqnarray*}
where 
\[
    C=\frac{\left(\frac{4n}{\delta^2}(1+\epsilon_1)\right)^2B^{\frac{8nt}{\delta^2}(\epsilon_2-\epsilon_1)}}{A^{\frac{8nt}{\delta^2}(1-\epsilon_1)}}.
\]
This shows boundedness and therefore completing the proof.
\end{proof}

For any $J\in H^s(T^n),~s\in\bbR$, define
\begin{eqnarray}
     \Ldel J = \sum_{k\in\Zn}m(\nu_k)\hat{J}_ke^{i \nu_k\cdot x}.\label{LJ}
\end{eqnarray}
 
\begin{lem}\label{lem_lu=v}
    Let $U$ and $V$ be as defined in \eqref{U} and \eqref{V} respectively, then $\Ldel U(t)=V(t)$.
\end{lem}
\begin{proof}
By \eqref{LJ}, we have
\begin{eqnarray*}
    \Ldel U(t)&=&\sum_{k\in\Zn}m(\nu_k)\hat{U}_k(t)e^{i \nu_k\cdot x}\\
    &=&\sum_{k\in\Zn}m(\nu_k)\hat{f}_ke^{m(\nu_k)t}e^{i\nu_k\cdot x}\\
    &=&V(t).
\end{eqnarray*}
\end{proof}

\begin{prop}\label{prop_du=v}
    Let $N\in \mathbb{N}\cup \{0\}$ and define
    \begin{eqnarray*}
        U^{(N)}(t)&=&\underset{k}{\sum} \hat{f}_km(\nu_k)^{N}e^{m(\nu_k)t}e^{i\nu_k\cdot x},\\
        V^{(N)}(t)&=&\underset{k}{\sum} \hat{f}_km(\nu_k)^{N+1}e^{m(\nu_k)t}e^{i\nu_k\cdot x}.
    \end{eqnarray*}
    Then $\frac{d}{dt}U^{(N)}(t)=V^{(N)}(t)$, for all $t\in (0,\infty).$ Equivalently,
    \[
        \frac{d^N}{dt^N}U(t) = \underset{k}{\sum} \hat{f}_km(\nu_k)^{N+1}e^{m(\nu_k)t}e^{i\nu_k\cdot x},
    \]
    where the differentiation here is in the sense of Gateaux differentiation.
\end{prop}
\begin{remark}
    We note that $U^{(0)}(t)=U(t)$ and $V^{(0)}(t)=V(t)$. In addition, similar to the argument in Theorem \ref{spatial_regularity}, for any $t\ge 0$, both $U^{(N)}(t)$ and $ V^{(N)}(t)$ are in $H^s(T^n)$ when $\beta\le n$ and both are in  $H^\infty(T^n)$ when $n<\beta<n+4$. 
\end{remark}
\begin{proof}
Let $t>0$, we show that $\frac{d}{dt}U^{(N)}(t)=V^{(N)}(t)$, where the differentiation is in the Gateaux sense, which is given by 
    \[
        \lim_{h\to 0}\left\|\frac{1}{h}\left[U^{(N)}(t+h)-U^{(N)}(t)\right]-V^{(N)}(t)\right\|^2_{H^q(T^n)}=0,
    \]
    where $q=s$ when $\beta\le n$ and $q$ is arbitrary when $n<\beta<n+4$. Equivalently, we  show that
    \begin{eqnarray}
        \lim_{h\to 0}\sum_{k\in \Zn}(1+\|k\|^2)^q|\hat{f}_k|^2m(\nu_k)^{2N}\left[\frac{1}{h}\left(e^{m(\nu_k)h}-1\right)e^{m(\nu_k)t}-m(\nu_k)e^{m(\nu_k)t}\right]^2=0.\label{gateaux}
    \end{eqnarray}
    This result follows from passing the limit inside the sum, which we justify next by the dominated convergence theorem. 
    
    When $\beta<n$, then $q=s$ and there exists a constant $C_1>0$ such that $|m(\nu_k)|<C_1$. Moreover, there exists $C_2>0$ such that 
    \begin{eqnarray}
        \left| \frac{e^{m(\nu_k)h}-1}{h}\right|< C_2,\label{exp_bound}
    \end{eqnarray}
    for all $k$, and for sufficiently small $h$. Combining this with the fact that $f\in H^s(T^n)$, it follows that the summand in the left hand side of \eqref{gateaux} is uniformly bounded.
    
When $\beta>n,$ then the expression
    \begin{eqnarray*}
        (1+\|k\|^2)^q|\hat{f}_k|^2m(\nu_k)^{2N}\left[\frac{1}{h}\left(e^{m(\nu_k)h}-1\right)e^{m(\nu_k)t}-m(\nu_k)e^{m(\nu_k)t}\right]^2,
    \end{eqnarray*}
    is uniformly bounded since $m(\nu_k)\to -\infty$ as $\|k\|\to \infty$.
    
    When $\beta=n$, then $q=s$ and it is sufficient to show that 
    \[
        m(\nu_k)^{2N}\left[\frac{1}{h}\left(e^{m(\nu_k)h}-1\right)e^{m(\nu_k)t}-m(\nu_k)e^{m(\nu_k)t}\right]^2=m(\nu_k)^{2N}e^{2m(\nu_k)t}\left[\frac{1}{h}\left(e^{m(\nu_k)h}-1\right)-m(\nu_k)\right]^2,
    \]
    is bounded. There exists $\epsilon>0$ such that
    \begin{eqnarray*}
        \left|\frac{e^{m(\nu_k)h}-1}{h}-m(\nu_k)\right| &\le& \left|\frac{e^{m(\nu_k)h}-1}{h}\right|+|m(\nu_k)|
        \le 2|m(\nu_k)|+\epsilon.
    \end{eqnarray*}
    Thus, by letting $\epsilon\to 0$ and by using \eqref{asymptotic_inequality} and \eqref{log-upperbound}, we have
    \begin{eqnarray*}
        m(\nu_k)^{2N}e^{2m(\nu_k)t}\left[\frac{1}{h}\left(e^{m(\nu_k)h}-1\right)-m(\nu_k)\right]^2&\le&4m(\nu_k)^{2(N+1)}e^{2m(\nu_k)t}\\
        &\le&\frac{4\left(\frac{4n}{\delta^2}(1+\epsilon_1)\right)^{2(N+1)}(B\|k\|^{2(N+1)})^{\frac{4nt(1-\epsilon_1)}{\delta^2(N+1)}}}{(A\|k\|)^{\frac{8nt}{\delta^2}(1-\epsilon_1)}}\\
        &=&\frac{4\left(\frac{4n}{\delta^2}(1+\epsilon_1)\right)^{2(N+1)}B^{\frac{4nt(1-\epsilon_1)}{\delta^2(N+1)}}}{A^{\frac{8nt}{\delta^2}(1-\epsilon_1)}},
    \end{eqnarray*}
    showing uniform boundedness and therefore completing the proof.
\end{proof}
 The following theorem summarizes the results in this subsection.
\begin{theorem}\label{exist_uniq}
    Let $f\in H^s(T^n)$, $\beta<n+4$, and $s\in\bbR$. Then, there exists a unique solution $U(t)$ to the nonlocal diffusion equation
    \begin{eqnarray}
        \begin{cases}
            \displaystyle\frac{dU}{dt} = \Ldel U(t),\\
            U(0) = f.
        \end{cases}\label{gateaux_nonlocal_diffusion_eqn}
    \end{eqnarray}
    Moreover, $U\in C^\infty((0,\infty);H^p(T^n))$, where $p$ is as defined in \eqref{p_values}.
\end{theorem}
\begin{remark}
 The time regularity in Theorem~\ref{exist_uniq} is in the sense of Gateaux differentiation.
\end{remark}
\begin{proof}
    The existence follows from Lemma \ref{lem_lu=v} and Proposition \ref{prop_du=v} by taking $N=0$. For the uniqueness, let $U_2(t)$ be another solution of \eqref{gateaux_nonlocal_diffusion_eqn}. We define $W(t)=U(t)-U_2(t)$. Then, $W(t)$ satisfies
$$
\begin{cases}
	\displaystyle\frac{dW}{dt} = \Ldel W,\\
	W(0)=0.
\end{cases}
$$
Represent $W(t)$ by its Fourier series
\begin{eqnarray*}
    W(t)=\sum_{k\in\Zn}\hat{W}_k(t)e^{i\nu_k\cdot x}.
\end{eqnarray*}
Lemma \ref{lem_lu=v} implies that
\[
\Ldel W(t) = \sum_{k\in\Zn} m(\nu_k)\hat{W}_k(t)e^{i\nu_k\cdot x},
\]
and
\[
\frac{dW}{dt}=\sum_{k\in\Zn} \frac{d\hat{W}_k}{dt}(t) e^{i\nu_k\cdot x}.
\]
From \eqref{gateaux_nonlocal_diffusion_eqn} and the uniqueness of Fourier coefficients, we have that
\[
    \frac{d\hat{W}_k}{dt}(t) = m(\nu_k)\hat{W}_k(t),
\]
for all $k$. This implies that $\hat{W}_k(t) = Ae^{m(\nu_k)t}$, where $A$ is a constant. Since $\hat{W}_k(0)=0$, then $\hat{W}_k(t) = 0$, for all $k$ which implies that $W(t)=0$. Therefore, $U(t)=U_2(t).$ The spatial regularity follows from Theorem \ref{spatial_regularity}.
\end{proof}

\subsection{Regular functions as solutions of the nonlocal diffusion equation}\label{regular_fns_as_sol}
In this section, we  focus on functions $f$ with absolutely summable Fourier coefficients, that is, $\underset{k\in\Zn}{\sum}|\hat{f}_k|<\infty$. The following theorem gives a class of functions that satisfy this condition, see \cite{grafakos2008classical}.

\begin{theorem}
    Let $s$ be a non-negative integer and let $0\le\alpha<1$. Assume that $f$ is a function defined on $T^n$ all of whose partial derivatives of order $s$ lie in the space of Holder continuous functions of order $\alpha$. Suppose that $s+\alpha>n/2$. Then $f$ has an absolutely convergent Fourier series.
\end{theorem}

\noindent Next we provide results on the temporal regularity of the nonlocal diffusion equation.

\begin{prop}\label{temporal_regularity}
Let $f\in H^s(T^n)$ such that $\underset{k\in\Zn}{\sum}|\hat{f}_k|<\infty$ and let $\beta< n+4$. Then, 
$$u(x,\cdot)\in C^\infty((0,\infty)),$$ 
for all $x\in T^n.$
\end{prop}
\begin{proof}
    We  use the Leibniz integral rule for  the counting measure to differentiate under the summation. Let $g_k(t)=\hat{f}_ke^{m(\nu_k)t}e^{i\nu_k\cdot x}$ and consider
    \begin{eqnarray*}
        \sum_{k\in \Zn}|g_k(t)|&=&\sum_{k\in\Zn} |\hat{f}_ke^{m(\nu_k)t}e^{i\nu_k\cdot x}|\\
        &=&\sum_{k\in\Zn} |\hat{f}_k|e^{m(\nu_k)t}\\
        &\le&\sum_{k\in\Zn}|\hat{f}_k|.
    \end{eqnarray*}
    Since $\underset{k\in\Zn}{\sum}|\hat{f}_k|<\infty$, then $g_k(t)$ is summable for any fixed $t$. Moreover, 
    \[
    \frac{dg_k}{dt}=\hat{f}_km(\nu_k)e^{m(\nu_k)t}e^{i\nu_k\cdot x},
    \]
    is continuous for all $k$. Now fix $t>0$, then there exist $\tau$ such that $0<\tau<t$. We define $\theta_k:=|\hat{f}_k||m(\nu_k)|e^{m(\nu_k)\tau}$. When $\beta<n$, then there exists $C>0$ such that $|m(\nu_k)|\le C$. Thus
    \[
        \theta_k=|\hat{f}_k||m(\nu_k)|e^{m(\nu_k)\tau}\le C|\hat{f}_k|,
    \]
    showing that $\theta_k$ is summable. When $\beta\ge n$, then $e^{m(\nu_k)\tau}\to 0$ as $\|k\|\to \infty$. Thus $|m(\nu_k)|e^{m(\nu_k)\tau}\le 1$ for sufficiently large $\|k\|$. Hence $\theta_k\le |\hat{f}_k|$, showing that $\theta_k$ is summable. Moreover
    \begin{eqnarray*}
        \left|\frac{dg_k}{dt}\right|&=& |\hat{f}_k||m(\nu_k)|e^{m(\nu_k)t}\\
        &\le& |\hat{f}_k||m(\nu_k)|e^{m(\nu_k)\tau}\\
        &=&\theta_k. 
    \end{eqnarray*}
    Therefore, 
    we can differentiate under the summation,
    \begin{eqnarray*}
        \frac{\partial u(x,t)}{\partial t}&=&\frac{d}{dt}\sum_{k\in\Zn}g_k(t)\\
        &=&\sum_{k\in\Zn}\frac{d g_k}{dt}\\
        &=&\sum_{k\in\Zn}\hat{f}_km(\nu_k)e^{m(\nu_k)t}e^{i\nu_k\cdot x}.
    \end{eqnarray*}
    For higher derivatives, we observe that 
    \[
        \frac{d^N g_k}{dt^N}=\hat{f}_k|m(\nu_k)|^Ne^{m(\nu_k)t}e^{i\nu_k\cdot x}.
    \]
    Define $\theta_k=|\hat{f}_k||m(\nu_k)|^Ne^{m(\nu_k)\tau}$. Then $\theta_k$ is summable by following similar arguments as above. Furthermore,
    \begin{eqnarray*}
        \left|\frac{d^N g_k}{dt^N}\right|&=&|\hat{f}_k||m(\nu_k)|^Ne^{m(\nu_k)t}\\
        &\le&|\hat{f}_k||m(\nu_k)|^Ne^{m(\nu_k)\tau}\\
        &=&\theta_k.
    \end{eqnarray*}
    This implies that $u(x,\cdot)$ is N-times continuously differentiable and 
    \begin{eqnarray*}
        \frac{\partial^N u(x,t)}{\partial t^N}&=&\frac{d^N}{dt^N}\sum_{k\in\Zn}g_k(t)\\
        &=&\sum_{k\in\Zn}\frac{d^N g_k}{dt^N}\\
        &=&\sum_{k\in\Zn}\hat{f}_k|m(\nu_k)|^Ne^{m(\nu_k)t}e^{i\nu_k\cdot x}.
    \end{eqnarray*}
    Since $N$ is arbitrary, it follows that $u(x,\cdot)\in C^\infty((0,\infty))$.
\end{proof}
From Theorem \ref{spatial_regularity} and Proposition \ref{temporal_regularity} we obtain the following regularity result.
\begin{theorem}
\label{thm:regularity1}
Let $n\ge 1,~\delta>0$, $\epsilon>0$ and $\beta< n+4$. Assume that  $f\in H^s(T^n)$ and its Fourier coefficients are summable. Then,
\begin{enumerate}
	\item $u \in C^\infty{((0,\infty);H^s(T^n))}$ for $\beta < n$,
	\item $u\in
	C^\infty\left((0,\infty);H^{s+\frac{4nt}{\delta^2}(1-\epsilon)}(T^n)\right)$ for $\beta=n$, 
	\item $u \in C^\infty{((0,\infty);H^\infty(T^n))}$ for $\beta>n$.
\end{enumerate}
\end{theorem}

The following lemma will be used to prove Theorem~\ref{convergence_NHE_Without_diffusionSource1} on the convergence of solutions of the nonlocal diffusion equation as $\delta\rightarrow 0^+$.
\begin{lem} \label{sec3.2_lem2}
	Let $n<\beta<n+2$ and $\delta\le 1$. Then, there exist $c_1>0$ and $c_2>0$ such that for all $\nu\in\bbR^n$,
	\begin{eqnarray*}
	     m^{\delta,\beta}(\nu) \le \max\{-c_1 \|\nu \|^{\beta-n}, -c_2 \|\nu\|^2\}.
	\end{eqnarray*}
\end{lem}
\begin{proof}
From Theorem \ref{thm:asymptotic_behavior}, we have
$$m^{1,\beta}(\nu)\sim c\|\nu\|^{\beta-n},$$
where
$$c=(2)2^{n+2-\beta}\frac{\Gamma\left(\frac{n+4-\beta}{2}\right)\Gamma\left(\frac{n+2}{2}\right)}{(\beta-n)\Gamma\left(\frac{\beta}{2}\right)}>0.$$
This is equivalent to 
\begin{eqnarray*}
    \lim_{\|\nu\|\to\infty}\frac{m^{1,\beta}(\nu)}{-\|\nu\|^{\beta-n}}=c,
\end{eqnarray*}
which implies that there is $c_1>0$ and $N>0$ such that for all $\|\nu\|>N$
\begin{eqnarray}
    m^{1,\beta}(\nu)\le-c_1\|\nu\|^{\beta-n}.\label{lem1_1}
\end{eqnarray}
On the other hand, from \cite{alali2020fourier}, we have 
\begin{eqnarray*}
    \lim_{\|\nu\|\to 0}\frac{m^{1,\beta}(\nu)}{-\|\nu\|^{\beta-n}}=1.
\end{eqnarray*}
Thus, there exists $c_2>0$ such that for all $\|\nu\|<N,$
\begin{eqnarray}
    m^{1,\beta}(\nu)\le-c_2\|\nu\|^2.\label{lem1_2}
\end{eqnarray}
Combining \eqref{lem1_1} and \eqref{lem1_2}, we have 
\begin{eqnarray}
    m^{1,\beta}(\nu)\le \max\{-c_1\|\nu\|^{\beta-n},-c_2\|\nu\|^2\}\label{lem1_3},
\end{eqnarray}
for all $\nu\in\Rn$. Using \eqref{lem1_3} and the fact that $m^{\delta,\beta}=\frac{1}{\delta^2}m^{1,\beta}(\delta\nu)$, we obtain
\begin{eqnarray*}
    m^{\delta,\beta}=\frac{1}{\delta^2}m^{1,\beta}(\delta\nu)&\le&\frac{1}{\delta^2}\max\{-c_1\|\delta\nu\|^{\beta-n},-c_2\|\delta\nu\|^2\}\\
    &=&\max\{-c_1\|\nu\|^{\beta-n}\delta^{\beta-(n+2)},-c_2\|\nu\|^2\}.
\end{eqnarray*}
Since $\delta\le 1$, then $-\delta^{\beta-(n+2)}\le -1$ and hence 
$$m^{\delta,\beta}(\nu)\le\max\{-c_1\|\nu\|^{\beta-n},-c_2\|\nu\|^2\}.$$
\end{proof}


Convergence of solutions of the nonlocal diffusion equation \eqref{eq:nonlocal_diffusion_nosource} to the solution of the corresponding classical diffusion equation is given next in Theorem~\ref{convergence_NHE_Without_diffusionSource1} and Theorem~\ref{convergence_NHE_Without_diffusionSource2}.
\begin{theorem}\label{convergence_NHE_Without_diffusionSource1}
	Let $n \ge 1, s\in\bbR$ and let $f \in H^s(T^n)$. Suppose $u$ is the solution of the classical diffusion equation $u_t = \Delta u$ with initial condition $ u|_{t=0}=f$. For any $\delta>0$, let $u^{\delta, \beta}$ be the solution of the nonlocal diffusion equation in \eqref{eq:nonlocal_diffusion_nosource}. Then, for $t>0$ and $\beta\le n$,
	\begin{eqnarray*}
		\lim_{\delta \to 0^+}u^{\delta, \beta}(\cdot,t)=u(\cdot,t),~~ \text{ in } H^s(T^n),
	\end{eqnarray*}
	and for $n<\beta\le n+2$,
	\begin{eqnarray*}
		\lim_{\delta \to 0^+}u^{\delta, \beta}(\cdot,t)=u(\cdot,t),~~ \text{ in } H^\infty(T^n).
	\end{eqnarray*}
\end{theorem}

\begin{proof}
The Fourier coefficients satisfy $\hat{u}_k^{\delta, \beta}= \hat{f}_ke^{m(\nu_k)t}$ and $\hat{u}_k= \hat{f}_ke^{-\|\nu_k \|^2t}$. When $\beta\le n$, then
\begin{eqnarray*}
	\| u^{\delta, \beta}(\cdot,t)-u(\cdot,t) \|^2_{H^s(T^n)}&=& \sum_{0 \ne k \in \mathbb{Z}^n}(1+ \|k \|^2)^s~ | \hat{u}_k^{\delta, \beta}-\hat{u}_k |^2 \\
	&=& \sum_{0 \ne k \in \mathbb{Z}^n}(1+ \|k \|^2)^s~ |e^{m(\nu_k)t}-e^{-\|\nu_k \|^2t} |^2 ~| \hat{f}_k |^2.
\end{eqnarray*}
To pass the limit $\delta\rightarrow 0^+$ inside the sum, it is sufficient to show that $|e^{m(\nu_k)t}-e^{-\|\nu_k \|^2t}|^2$ as a function of $k$ is uniformly bounded. 
 Using  \eqref{multipliers}, it is straightforward to see that  $m(\nu)\le 0$ for $\nu\in\mathbb{R}^n$. Thus, 
\begin{eqnarray*}
    \left|e^{m(\nu_k)t}-e^{-\|\nu_k \|^2t}\right|^2\le  (e^{m(\nu_k)t}+e^{-\|\nu_k \|^2t})^2\le 4.
\end{eqnarray*}
Since, $\underset{\delta\to 0^+}{\lim}m(\nu_k)=-\|\nu_k\|^2$, then 
\begin{eqnarray*}
	\lim_{\delta \to 0^+}u^{\delta, \beta}(\cdot,t)=u(\cdot,t),~~ \text{ in } H^s(T^n).
\end{eqnarray*}
For the case $\beta>n$, fix an arbitrary  $p\in\bbR$.  Then,
\begin{eqnarray*}
	\| u^{\delta, \beta}(\cdot,t)-u(\cdot,t) \|^2_{H^p(T^n)}&=& \sum_{0 \ne k \in \mathbb{Z}^n}(1+ \|k \|^2)^p| \hat{u}_k^{\delta, \beta}-\hat{u}_k |^2 \\
	&=& \sum_{0 \ne k \in \mathbb{Z}^n}(1+ \|k \|^2)^s\, (1+ \|k \|^2)^{p-s}~ |e^{m(\nu_k)t}-e^{-\|\nu_k \|^2t} |^2 ~| \hat{f}_k |^2.
\end{eqnarray*}
Since $(1+ \|k \|^2)^s |\hat{f}_k |^2$ is summable, then 
to pass the limit inside the above sum, we  show that the following  function in $k$ that is given by  
\[
(1+ \|k \|^2)^{p-s}|e^{m(\nu_k)t}-e^{-\|\nu_k \|^2t}|^2,
\]
is uniformly bounded. 
First, we  rewrite the above expression as 
\begin{eqnarray*}
	(1+ \|k \|^2)^{p-s}|e^{m(\nu_k)t}-e^{-\|\nu_k \|^2t}|^2&=&(1+ \|k \|^2)^{p-s}~e^{2m(\nu_k)t}\left(1-e^{-(m(\nu_k)+ \|\nu_k \|^2)t}\right)^2.
\end{eqnarray*}
Then, we observe that 
\begin{eqnarray*}
    \mdel(\nu_k)+\|\nu_k\|^2&=&-\|\nu_k\|^2~_2F_3\left(1,\frac{n+2-\beta}{2};2,\frac{n+2}{2},\frac{n+4-\beta}{2};-\frac{1}{4}\|\nu_k\|^2\delta^2\right)+\|\nu_k\|^2\\
    &=&\|\nu_k\|^2\left(1-_2F_3\left(1,\frac{n+2-\beta}{2};2,\frac{n+2}{2},\frac{n+4-\beta}{2};-\frac{1}{4}\|\nu_k\|^2\delta^2\right)\right).
\end{eqnarray*}
Since $_2F_3\left(1,\frac{n+2-\beta}{2};2,\frac{n+2}{2},\frac{n+4-\beta}{2};x\right)\le 1$ for all $x\le 0$, then $\mdel + \|\nu_k\|^2\ge 0$. Therefore,
$$1-e^{-(m(\nu_k)+ \|\nu_k \|^2)t}<1.$$ 
Using this fact, we have 
\begin{eqnarray*}
	(1+ \|k \|^2)^{p-s}\left|e^{m(\nu_k)t}-e^{-\|\nu_k \|^2}\right|&=&(1+ \|k \|^2)^{p-s}~e^{2m(\nu_k)t}(1-e^{-(m(\nu_k)+ \|\nu_k \|^2t})^2 \\
	&<&(1+ \|k \|^2)^{p-s}~e^{2m(\nu_k)t}.
	\end{eqnarray*}
 Lemma \ref{sec3.2_lem2} implies that there exist $c_1>0$ and $c_2>0$ such that  
 \[
 e^{m(\nu_k) t}\le \max\{e^{-c_1 \|\nu_k\|^{\beta-n}\; t}, e^{-c_2 \|\nu_k\|^{2} t}\}. 
 \]
 Consequently, 
\[
(1+ \|k \|^2)^{p-s}~e^{2m(\nu_k)t}
	\le \frac{(1+ \|k \|^2)^{p-s}}{\min\{\exp(2c_1 \|\nu_k\|^{\beta-n}\; t), \exp(2c_2 \|\nu_k\|^{2} t)\}}, 
\]
which is bounded for sufficiently large $k$ for all $\delta\in[0,1]$. 
\end{proof}

\begin{theorem}\label{convergence_NHE_Without_diffusionSource2}
Let $n\ge 1, s\in\bbR$ and let $f\in H^s(T^n)$. Suppose $u$ is the solution of the classical diffusion equation $u_t = \Delta u$, with $u|_{t=0}=f$, and for any $\beta<n+4$, let $\udel$ be the solution of the nonlocal diffusion equation \eqref{eq:nonlocal_diffusion_nosource}, then for $t>0$
\[
    \lim_{\beta\to n+2}\udel(\cdot,t) = u(\cdot,t), \text{ in } H^\infty(T^n).
\]
\end{theorem}
\begin{proof}
Let $q\in\mathbb{R}$ be arbitrary. Consider
\begin{eqnarray*}
    \|\udel(\cdot,t) -u(\cdot,t)\|^2_{H^q(T^n)} = \sum_{ k\in\Zn}(1+\|k\|^2)^{q-s}\left|e^{m(\nu_k)t}-e^{-\|\nu_k\|^2t}\right|^2(1+\|k\|^2)^s\|\hat{f}_k\|^2.
\end{eqnarray*}
We observe that for $\beta$ near $n+2$, $m(\nu_k)\rightarrow -\infty$ as $\|k\|\rightarrow \infty$. Thus,
we can pass the limit in $\beta$ inside the sum above, since $f\in H^s(T^n)$ and the expression
\[
    (1+\|k\|^2)^{q-s}\left|e^{m(\nu_k)t}-e^{-\|\nu_k\|^2t}\right|^2,
\]
is uniformly bounded. 
Since $\underset{\beta\to n+2}{\lim}m(\nu_k)=-\|\nu_k\|^2$, then the result follows.
\end{proof}

\section{Nonlocal diffusion equation with a diffusion source}\label{NHE_WithdiffusionSource}
In this section, we focus on the following nonlocal diffusion equation with a diffusion source and zero initial data
\begin{eqnarray}
\label{eq:nonlocal_diffusion_source}
\begin{cases}
    u_t(x,t)=\Ldel u(x,t) +b(x) ,~~x\in T^n,~t>0,\\
    u(x,0) = 0.
\end{cases}\label{nonlocal_diffusion_eqn_src}
\end{eqnarray}
In order to study the existence, uniqueness, and regularity of solutions to \eqref{nonlocal_diffusion_eqn_src} over the space of periodic distributions, we consider the identification $U(t)=u(\cdot,t)$, with $U:[0,\infty)\rightarrow H^s(T^n)$.

Let $b\in H^s(T^n)$ and define $U, V:[0,\infty)\to H^q(T^n)$ for some $q\in\mathbb{R}$, by
\begin{eqnarray}
    U(t)&=&\hat{b}_0t + \sum_{0\ne k\in\Zn}\frac{e^{m(\nu_k)t}-1}{m(\nu_k)}\hat{b}_ke^{i\nu_k\cdot x},\label{U_2}\\[1ex]
    V(t)&=&\sum_{k\in\Zn} e^{m(\nu_k)t} \hat{b}_k e^{i\nu_k\cdot x}.\label{V_2}
\end{eqnarray}
Observe that for any $t\ge0$, $U(t)$ and $V(t)$ are well-defined periodic distributions, since $\frac{e^{m(\nu_k)t}-1}{m(\nu_k)}$ and $e^{m(\nu_k)t}$ are bounded functions in $k$.

\begin{theorem}\label{U_2spatial_regularity}
    Let $n\ge 1,~\delta>0$, and $\beta< n+4$.  Assume that $\epsilon_1>0$ and $b\in H^s(T^n)$ for some $s\in\bbR$. Then for any fixed $t>0$, $U(t)\in H^p(T^n)$ and $V(t)\in H^{r}(T^n)$, where
    \begin{eqnarray}
	p&=&\begin{cases}
	s, \text{ if }\beta\le n,\\
	s+\beta-n, \text{ if } \beta> n, 
	\end{cases}~~\mbox { and }\quad
	r=\begin{cases}
	s, \text{ if }\beta< n,\\
	s+\dfrac{4nt}{\delta^2}(1-\epsilon_1),\text{ if }\beta=n,\\
	\infty, \text{ if } \beta> n.
	\end{cases}\label{U_2p_values}
\end{eqnarray}
\end{theorem}

\begin{proof}
We observe that
\begin{eqnarray*}
    \|U(t)-\hat{b}_0t\|^2_{H^p(T^n)}&=&\sum_{0\ne k\in\Zn}(1+\|k\|^2)^p|\hat{U}_k(t)|^2\\
    &=&\sum_{0\ne k\in\Zn}\frac{(1+\|k\|^2)^{p-s}}{|m(\nu_k)|^2}(e^{m(\nu_k)}-1)^2(1+\|k\|^2)^s|\hat{b}_k|^2.
\end{eqnarray*}
Since $b\in H^s(T^n)$ and $e^{m(\nu_k)t}$ is bounded because $m(\nu_k)< 0$, then the result follows by showing that
\[
    \frac{(1+\|k\|^2)^{p-s}}{|m(\nu_k)|^2},
\]
is bounded for $k\ne 0$. When $\beta\le n$, then $p-s=0$ and by using Theorem \ref{thm:asymptotic_behavior}, there exist $C_1>0$ and $r_1>0$ such that $|m(\nu_k)|\ge C_1$, for all $\|k\|\ge r_1$. Thus,
\[
    \frac{(1+\|k\|^2)^{p-s}}{|m(\nu_k)|^2}\le\frac{1}{C_1^2}.
\]
When $\beta>n$, then $p-s=\beta-n$, and by using Theorem \ref{thm:asymptotic_behavior}, there exist $C_2>0$ and $r_2>0$ such that $|m(\nu_k)|\ge C_2\|k\|^{\beta-n}$, for all $\|k\|\ge r_2$. This implies that
\[
    \frac{(1+\|k\|^2)^{p-s}}{|m(\nu_k)|^2}\le\frac{1}{C_2^2}\left(\frac{1+\|k\|^2}{\|k\|^2}\right)^{\beta-n},
\]
which is bounded. The proof of $V(t)\in H^r(T^n)$ is similar to the proof of Theorem \ref{spatial_regularity}.
\end{proof}

\begin{lem}\label{lem_lu+b=v}
    Let $U$ and $V$ be as defined in \eqref{U_2} and \eqref{V_2},  respectively. Then 
    \[V(t)=\Ldel U(t)+b.
    \]
\end{lem}
\begin{proof}
By using \eqref{LJ}, for any $x\in T^n$ and $t>0$, we have
\begin{eqnarray*}
    \Ldel U(t)(x)&=&\Ldel(\hat{b}_0t)+\sum_{0\ne k\in\Zn}m(\nu_k)\hat{U}_k(t)e^{i \nu_k\cdot x}\\
    &=&\sum_{0\ne k\in\Zn}m(\nu_k)\frac{(e^{m(\nu_k)t}-1)}{m(\nu_k)}\hat{b}_ke^{i\nu_k\cdot x}\\
    &=&\sum_{0\ne k\in\Zn}\hat{b}_ke^{m(\nu_k)t}e^{i\nu_k\cdot x}-\sum_{0\ne k\in\Zn}\hat{b}_ke^{i\nu_k\cdot x}\\
    &=&V(t)(x)-b(x).
\end{eqnarray*}
\end{proof}

\begin{prop}\label{prop_du_2=v}
    Let $U(t)$ and $V(t)$ be as defined in \eqref{U_2} and \eqref{V_2}, respectively. Then, 
    \[
        \frac{dU}{dt} = V(t).
    \]
    Moreover, for $N\ge 1$,
    \[
        \frac{d^N U}{dt^N} = \underset{k\in\Zn}{\sum} \hat{b}_km(\nu_k)^{N-1}e^{m(\nu_k)t}e^{i\nu_k\cdot x},
    \]
    for all $t\in (0,\infty)$, where the differentiation here is in the sense of Gateaux differentiation.
\end{prop}

\begin{proof}
    We  show that
    \[
        \lim_{h\to 0}\left\|\frac{1}{h}\left[U(t+h)-U(t)\right]-V(t)\right\|^2_{H^q(T^n)}=0,
    \]
    where $q=s$ when $\beta\le n$ and $q$ is arbitrary when $\beta>n$. Equivalently, we  show that
    \begin{eqnarray}
        \lim_{h\to 0}\sum_{0\ne k\in \Zn}(1+\|k\|^2)^q|\hat{b}_k|^2\left[\frac{1}{h}\left(e^{m(\nu_k)h}-1\right)\frac{e^{m(\nu_k)t}}{m(\nu_k)}-e^{m(\nu_k)t}\right]^2=0.\label{U_2gateaux}
    \end{eqnarray}
    This result follows from passing the limit inside the sum, which we justify next by the dominated convergence theorem. 
  
  When $\beta<n$, then $q=s$ and similar to \eqref{exp_bound}, there exists a constant $C_2>0$ such that 
    \[
        \left| \frac{e^{m(\nu_k)h}-1}{h}\right|<C_2,
    \]
    for all $k$, and for sufficiently small $h$. Moreover, $e^{m(\nu_k)t}$ and $\frac{1}{m(\nu_k)},$ for $k\ne 0,$ are bounded. Combining this with the fact that $b\in H^s(T^n)$, it follows that the summand in the left hand side of \eqref{U_2gateaux} is uniformly bounded.
    
    When $\beta>n,$ then the expression
    \begin{eqnarray*}
        (1+\|k\|^2)^q|\hat{b}_k|^2\left[\frac{1}{h}\left(e^{m(\nu_k)h}-1\right)\frac{e^{m(\nu_k)t}}{m(\nu_k)}-e^{m(\nu_k)t}\right]^2
    \end{eqnarray*}
    is uniformly bounded since $m(\nu_k)\to -\infty$ as $\|k\|\to \infty$.
    
    When $\beta=n$, then $q=s$ and thus it is sufficient to show that 
    \[
        \left[\frac{1}{h}\left(e^{m(\nu_k)h}-1\right)\frac{e^{m(\nu_k)t}}{m(\nu_k)}-e^{m(\nu_k)t}\right]^2=\frac{e^{2m(\nu_k)t}}{m(\nu_k)^2}\left[\frac{1}{h}\left(e^{m(\nu_k)h}-1\right)-m(\nu_k)\right]^2
    \]
    is uniformly bounded. There exists $\epsilon>0$ such that
    \begin{eqnarray*}
        \left|\frac{e^{m(\nu_k)h}-1}{h}-m(\nu_k)\right| &\le& \left|\frac{e^{m(\nu_k)h}-1}{h}\right|+|m(\nu_k)|
        \le 2|m(\nu_k)|+\epsilon.
    \end{eqnarray*}
    Thus, by letting $\epsilon\to 0$, we have
    \begin{eqnarray*}
        \frac{e^{2m(\nu_k)t}}{m(\nu_k)^2}\left[\frac{1}{h}\left(e^{m(\nu_k)h}-1\right)-m(\nu_k)\right]^2&\le&4e^{2m(\nu_k)t},
    \end{eqnarray*}
    showing boundedness since $m(\nu_k)\le 0$, and therefore completing the proof of the first part.
    The second part of this proposition follows from arguments similar to those in the proof of  Proposition~\ref{prop_du=v}.
\end{proof}

The following regularity theorem summarizes the results of this subsection.
\begin{theorem}
    Let $b\in H^s(T^n)$ with $s\in\bbR$. Then there exists a unique solution $U$ to the nonlocal diffusion equation
    \begin{eqnarray}
        \begin{cases}
            \displaystyle\frac{dU}{dt}=\Ldel U(t) +b,\\
            U(0) = 0.
        \end{cases}\label{U_2gateaux_nonlocal_diffusion_eqn}
    \end{eqnarray}
    Moreover, $U\in C^\infty((0,\infty);H^p(T^n))$, where $p$ is as defined in \eqref{U_2p_values}.
\end{theorem}
\begin{remark}
 The temporal regularity is in the sense of Gateaux differentiation.
\end{remark}

\begin{proof}
    The existence follows from Lemma \ref{lem_lu+b=v} and Proposition \ref{prop_du_2=v}. For the uniqueness, the proof is similar to the proof of uniqueness in Theorem \ref{exist_uniq}.  The spatial regularity follows from Theorem \ref{U_2spatial_regularity} and the temporal regularity follows from Proposition~\ref{prop_du_2=v}. 
\end{proof}

\subsection{Regular functions as solutions of nonlocal diffusion equation with diffusion source}\label{regular_fns_as_sol2}
In this section, we  focus on functions $b$ with absolutely summable Fourier coefficients,  $\underset{k\in\Zn}{\sum}|\hat{b}_k|<\infty$.

\begin{prop}\label{U_2temporal_regularity}
Let $b\in H^s(T^n)$ such that$\underset{k\in\Zn}{\sum}|\hat{b}_k|<\infty$ and let $\beta< n+4$. Then 
$$u(x,\cdot)\in C^\infty((0,\infty)),$$ 
for all $x\in T^n.$
\end{prop}

\begin{proof}
    We  use the Leibniz rule to differentiate under the sum. Let $g_k(t)=\frac{e^{m(\nu_k)t}-1}{m(\nu_k)}\hat{b}_ke^{i\nu_k\cdot x}$ and consider
    \begin{eqnarray*}
        \sum_{0\ne k\in \Zn}|g_k(t)|&=&\sum_{0\ne k\in\Zn} \left|\frac{e^{m(\nu_k)t}-1}{m(\nu_k)}\hat{b}_ke^{i\nu_k\cdot x}\right|\\
        &=&\sum_{0\ne k\in\Zn} |\hat{b}_k|\left|\frac{e^{m(\nu_k)t}-1}{m(\nu_k)}\right|\\
        &\le&\sum_{0\ne k\in\Zn} |\hat{b}_k|\frac{1}{|m(\nu_k)|},
    \end{eqnarray*}
    where in the last inequality, we used the fact that $m(\nu_k)<0$. Since $\frac{1}{|m(\nu_k)|},~k\ne 0,$ is bounded and $\underset{k\in\Zn}{\sum}|\hat{b}_k|<\infty$, then $g_k(t)$ is summable for any fixed $t$. Moreover, 
    \[
    \frac{dg_k}{dt}=\hat{b}_ke^{m(\nu_k)t}e^{i\nu_k\cdot x},
    \]
    is continuous for all $k$. Now fix $t>0$, then there exists $\tau$ such that $0<\tau<t$. We define $\theta_k:=|\hat{b}_k|e^{m(\nu_k)\tau}$. Since $m(\nu_k)\le 0$, for all $k$, then
    \[
        \theta_k=|\hat{b}_k|e^{m(\nu_k)\tau}\le |\hat{b}_k|,
    \]
    showing that $\theta_k$ is summable. Moreover,
    \begin{eqnarray*}
        \left|\frac{dg_k}{dt}\right|&=& |\hat{b}_k|e^{m(\nu_k)t}\\
        &\le& |\hat{b}_k|e^{m(\nu_k)\tau}\\
        &=&\theta_k. 
    \end{eqnarray*}
    Therefore, 
    \begin{eqnarray*}
        \frac{\partial u(x,t)}{\partial t}&=&\hat{b}_0+\frac{d}{dt}\sum_{0\ne k\in\Zn}g_k(t)\\
        &=&\hat{b}_0+\sum_{0\ne k\in\Zn}\frac{dg_k}{dt}\\
        &=&\sum_{k\in\Zn}\hat{b}_ke^{m(\nu_k)t}e^{i\nu_k\cdot x}.
    \end{eqnarray*}
    This shows that $u$ is differentiable with respect to $t$. For higher derivatives, let $N\ge 2$ be an integer, we observe that 
    \[
        \frac{d^N g_k}{dt^N}=\hat{b}_k\;(m(\nu_k))^{N-1}e^{m(\nu_k)t}e^{i\nu_k\cdot x}.
    \]
    Define $\theta_k=|\hat{b}_k||m(\nu_k)|^{N-1}e^{m(\nu_k)\tau}$. When $\beta<n$, then there exists $C>0$ such that $|m(\nu_k)|\le C$. Thus,
    \[
        \theta_k=|\hat{b}_k||m(\nu_k)|^{N-1}e^{m(\nu_k)\tau}\le C^{N-1}|\hat{b}_k|,
    \]
    showing that $\theta_k$ is summable. When $\beta\ge n$, then $e^{m(\nu_k)\tau}\to 0$ as $\|k\|\to \infty$. Thus, $|m(\nu_k)|^{N-1}e^{m(\nu_k)\tau}\le~1$ for sufficiently large $\|k\|$. Hence $\theta_k\le |\hat{b}_k|$, showing that $\theta_k$ is summable. Furthermore,
    \begin{eqnarray*}
        \left|\frac{d^N g_k}{dt^N}\right|&=&|\hat{b}_k||m(\nu_k)|^{N-1}e^{m(\nu_k)t}\\
        &\le&|\hat{b}_k||m(\nu_k)|^{N-1}e^{m(\nu_k)\tau}\\
        &=&\theta_k.
    \end{eqnarray*}
    This implies that $u(x,\cdot)$ is N-times continuously differentiable and 
    \begin{eqnarray*}
        \frac{\partial^N u(x,t)}{\partial t^N}&=&\frac{d^N}{dt^N}\sum_{k\in\Zn}g_k(t)\\
        &=&\sum_{k\in\Zn}\frac{d^N g_k}{dt^N}\\
        &=&\sum_{k\in\Zn}\hat{b}_k|m(\nu_k)|^{N-1}e^{m(\nu_k)t}e^{i\nu_k\cdot x}.
    \end{eqnarray*}
    Since $N$ is arbitrary, it follows that $u(x,\cdot)\in C^\infty((0,\infty))$.
\end{proof}

From Theorem \ref{U_2spatial_regularity} and Proposition \ref{U_2temporal_regularity} we obtain the following regularity result.
\begin{theorem}\label{thm:regularity_diffusionSource1}
Let $n\ge 1,~\delta>0$ and $\beta< n+4$. Assume that  $b\in H^s(T^n)$ and its Fourier coefficients are summable. Then,
\begin{enumerate}
	\item $u \in C^\infty{((0,\infty);H^s(T^n))}$, for $\beta \le n$,
	\item $u\in
	C^\infty\left((0,\infty);H^{s+\beta-n}(T^n)\right)$, for $\beta>n$.
\end{enumerate}
\end{theorem}
Convergence of solutions of the nonlocal diffusion equation \eqref{eq:nonlocal_diffusion_source} to the solution of the corresponding classical diffusion equation is given next in Theorem~\ref{convergence_NHE_With_diffusionSource1} and Theorem~\ref{convergence_NHE_With_diffusionSource2}.
\begin{theorem}
\label{convergence_NHE_With_diffusionSource1}
	Let $n \ge 1$ and  $b \in H^s(T^n)$, with $s\in\bbR$. Suppose $u$ is the solution of the classical diffusion equation $u_t = \triangle u + b$ with initial condition $ u|_{t=0}=0$. For any $\delta>0$, let $u^{\delta, \beta}$ be the solution of the nonlocal diffusion equation
	\begin{eqnarray}
    \begin{cases}
        \udel_t(x,t)=\Ldel \udel(x,t) +b(x),~~x\in T^n,~t>0,\\
        \udel(x,0) = 0, ~~x\in T^n.
    \end{cases}\label{nonlocal_diffusion_eqn2}
    \end{eqnarray}
	Then, for $t>0$ and  $\beta\le n$,
	\begin{eqnarray*}
		\lim_{\delta \to 0^+}u^{\delta, \beta}(\cdot,t)=u(\cdot,t),~~ \text{ in } H^s(T^n),
	\end{eqnarray*}
	and for $n<\beta\le n+2$,
	\begin{eqnarray*}
		\lim_{\delta \to 0^+}u^{\delta, \beta}(\cdot,t)=u(\cdot,t),~~ \text{ in } H^{s+\beta-n}(T^n).
	\end{eqnarray*}
\end{theorem}

\begin{proof}
    The Fourier coefficients satisfy $\hat{u}_k^{\delta, \beta}=\frac{(e^{m(\nu_k)t}-1)}{m(\nu_k)}\hat{b}_k$ and $\hat{u}_k= \frac{(e^{-\|\nu_k\|^2t}-1)}{-\|\nu_k\|^2}\hat{b}_k$, for $k\ne 0$. For $\beta\le n$, then
\begin{eqnarray*}
	\| u^{\delta, \beta}(\cdot,t)-u(\cdot,t) \|^2_{H^s(T^n)}&=& \sum_{0 \ne k \in \mathbb{Z}^n}(1+ \|k \|^2)^s~ | \hat{u}_k^{\delta, \beta}-\hat{u}_k |^2 \\
	&=& \sum_{0 \ne k \in \mathbb{Z}^n}(1+ \|k \|^2)^s~ \left|\frac{(e^{m(\nu_k)t}-1)}{m(\nu_k)}-\frac{(e^{-\|\nu_k\|^2t}-1)}{-\|\nu_k\|^2} \right|^2 ~| \hat{b}_k |^2.
\end{eqnarray*}
To pass the limit $\delta\rightarrow 0^+$ inside the sum, it is sufficient to show that $\left|\frac{(e^{m(\nu_k)t}-1)}{m(\nu_k)}-\frac{(e^{-\|\nu_k\|^2t}-1)}{-\|\nu_k\|^2} \right|$  is uniformly bounded. 
Using  \eqref{multipliers},   $m(\nu)\le 0$ for $\nu\in\mathbb{R}^n$, and  thus, 
\begin{eqnarray*}
    \left|\frac{(e^{m(\nu_k)t}-1)}{m(\nu_k)}-\frac{(e^{-\|\nu_k\|^2t}-1)}{-\|\nu_k\|^2} \right|
    &\le&\frac{1}{|m(\nu_k)|}+\frac{1}{\|\nu_k\|^2}.
\end{eqnarray*}
Since $\frac{1}{m(\nu_k)},~k\ne 0$, and $\frac{1}{\|\nu_k\|^2},~k\ne 0$, are bounded, then $\left|\frac{(e^{m(\nu_k)t}-1)}{m(\nu_k)}-\frac{(e^{-\|\nu_k\|^2t}-1)}{-\|\nu_k\|^2} \right|$ is uniformly bounded. 
For the case $\beta>n$, consider
\begin{eqnarray*}
	\| u^{\delta, \beta}(\cdot,t)-u(\cdot,t) \|^2_{H^{s+\beta-n}(T^n)}&\!\!\!\!=\!\!\!\!& \sum_{0 \ne k \in \mathbb{Z}^n}(1+ \|k \|^2)^{s+\beta-n}~ | \hat{u}_k^{\delta, \beta}-\hat{u}_k |^2 \\
	&=& \sum_{0 \ne k \in \mathbb{Z}^n}(1+ \|k \|^2)^s(1+ \|k \|^2)^{\beta-n}~ \!\!\left|\frac{(e^{m(\nu_k)t}-1)}{m(\nu_k)}\!-\!\frac{(e^{-\|\nu_k\|^2t}-1)}{-\|\nu_k\|^2} \right|^2 \!\!| \hat{b}_k |^2.
\end{eqnarray*}
Since $(1+ \|k \|^2)^s |\hat{b}_k |^2$ is summable, then 
to pass the limit inside the above sum, we  show that the following  function in $k$ that is given by  
\[
(1+ \|k \|^2)^{\beta-n}~ \left|\frac{(e^{m(\nu_k)t}-1)}{m(\nu_k)}-\frac{(e^{-\|\nu_k\|^2t}-1)}{-\|\nu_k\|^2} \right|^2
\]
is uniformly bounded. 
Since $m(\nu)\le 0$ for all $\nu\in\Rn$, then
\begin{eqnarray*}
    (1+ \|k \|^2)^{\beta-n}~ \left|\frac{(e^{m(\nu_k)t}-1)}{m(\nu_k)}-\frac{(e^{-\|\nu_k\|^2t}-1)}{-\|\nu_k\|^2} \right|^2
    &\le&(1+ \|k \|^2)^{\beta-n}~ \left(\frac{1}{|m(\nu_k)|}+\frac{1}{\|\nu_k\|^2} \right)^2.
\end{eqnarray*}
By using Theorem \ref{thm:asymptotic_behavior}, there exists constant $C>0$ such that $|m(\nu_k)|>C\|k\|^{\beta-n}$. Furthermore, using the fact that $A\|k\|\le\|\nu_k\|\le B\|k\|$ for positive constants $A$ and $B$ and since $\beta-n\le 2$, then $\|\nu_k\|^2\ge A^2\|k\|^{\beta-n}$. Therefore,
\begin{eqnarray*}
    (1+ \|k \|^2)^{\beta-n}~ \left(\frac{1}{|m(\nu_k)|}+\frac{1}{\|\nu_k\|^2} \right)^2&\le& (1+\|k\|^2)^{\beta-n}\left(\frac{1}{C\|k\|^{\beta-n}}+\frac{1}{A^2\|k\|^{\beta-n}}\right)^2\\[1ex]
    &\le&\max\left({\frac{1}{C^2},\frac{1}{A^4}}\right)\left(\frac{1+\|k\|^2}{\|k\|^2}\right)^{\beta-n},
\end{eqnarray*}
showing uniform boundedness. Whether $\beta\le n$ or $n<\beta\le n+2$, we have $\underset{\delta\to 0^+}{\lim}m(\nu_k)=-\|\nu_k\|^2$, which implies that $\underset{\delta\to 0^+}{\lim}\|u^{\delta,\beta}(\cdot,t)-u(\cdot,t)\|_{H^s(T^n)}=0$, or $\underset{\delta\to 0^+}{\lim}\|u^{\delta,\beta}(\cdot,t)-u(\cdot,t)\|_{H^{s+\beta-n}(T^n)}=0$, respectively, and therefore completing the proof. 
\end{proof}
A proof of the following lemma on the monotonicity of the multipliers can be found in \cite{alali2021fourier}. 

\begin{lem}\label{betabetaprime}
    Let $\beta'<\beta\le n+2$. Then, for all $\nu\ne 0$, $\mdel(\nu) < m^{\delta,\beta'}(\nu)$.
\end{lem}
\begin{theorem}\label{convergence_NHE_With_diffusionSource2}
Let $n\ge 1, s\in\bbR$ and let $b\in H^s(T^n)$. Suppose $u$ is the solution of the classical diffusion equation $u_t = \Delta u+b$, with $u|_{t=0}=0$, and for any $\beta<n+2$, let $\udel$ be the solution of the nonlocal diffusion equation \eqref{nonlocal_diffusion_eqn2}. Then, for $t>0$ and $0<\epsilon<2$,
\[
    \lim_{\beta\to (n+2)^-}\udel(\cdot,t) = u(\cdot,t), \text{ in } H^{s+2-\epsilon}(T^n).
\]
\end{theorem}
\begin{proof}
For $0<\epsilon< 2$, define $\beta'=n+2-\epsilon$. For any $\beta>\beta'$, we have from Theorem~\ref{U_2spatial_regularity} that $\udel\in H^{s+\beta-n}(T^n)\subset H^{s+2-\epsilon}(T^n)$. Furthermore, $u\in H^{s+2}(T^n)\subset H^{s+2-\epsilon}(T^n)$. Thus the limit makes sense. Consider
\begin{eqnarray*}
    \|\udel(\cdot,t) -u(\cdot,t)\|^2_{H^{s+2-\epsilon}(T^n)} = \sum_{0\ne k\in\Zn}(1+\|k\|^2)^{2-\epsilon}\left|\frac{e^{\mdel(\nu_k)t}-1}{\mdel(\nu_k)}-\frac{e^{-\|\nu_k\|^2t}-1}{-\|\nu_k\|^2}\right|^2(1+\|k\|^2)^s\|\hat{b}_k\|^2.
\end{eqnarray*}
Since $b\in H^s(T^n)$, in order to pass the limit in $\beta$ inside the sum, we show that the expression
\[
    (1+\|k\|^2)^{2-\epsilon}\left|\frac{e^{\mdel(\nu_k)t}-1}{\mdel(\nu_k)}-\frac{e^{-\|\nu_k\|^2t}-1}{-\|\nu_k\|^2}\right|^2
\]
is uniformly bounded for $k\ne 0$ and $\beta\in [\beta',n+2)$. Applying Lemma \ref{betabetaprime},
\begin{eqnarray*}
    (1+\|k\|^2)^{2-\epsilon}\left|\frac{e^{\mdel(\nu_k)t}-1}{\mdel(\nu_k)}-\frac{e^{-\|\nu_k\|^2t}-1}{-\|\nu_k\|^2}\right|^2&\le&(1+\|k\|^2)^{2-\epsilon}\left(\frac{1}{|\mdel(\nu_k)|}+\frac{1}{\|\nu_k\|^2}\right)^2\\
    &\le&(1+\|k\|^2)^{2-\epsilon}\left(\frac{1}{|m^{\delta,\beta'}(\nu_k)|}+\frac{1}{\|\nu_k\|^2}\right)^2.
\end{eqnarray*}
From Theorem~\ref{thm:asymptotic_behavior}, there exists $C>0$ such that $|m^{\delta,\beta'}(\nu_k)|\ge C\|k\|^{2-\epsilon}$. Furthermore, there exists $A>0$ such that $\|\nu_k\|\ge A\|k\|$. Thus,
\begin{eqnarray*}
    (1+\|k\|^2)^{2-\epsilon}\left(\frac{1}{|m^{\delta,\beta'}(\nu_k)|}+\frac{1}{\|\nu_k\|^2}\right)^2&\le& (1+\|k\|^2)^{2-\epsilon}\left(\frac{1}{C\|k\|^{2-\epsilon}}+\frac{1}{A^2\|\nu_k\|^2}\right)^2\\
    &\le& (1+\|k\|^2)^{2-\epsilon}\left(\frac{1}{C\|k\|^{2-\epsilon}}+\frac{1}{A^2\|\nu_k\|^{2-\epsilon}}\right)^2,
\end{eqnarray*}
which is uniformly bounded. Since $\underset{\beta\to (n+2)^-}{\lim}m(\nu_k)=-\|\nu_k\|^2$, then the result follows.
\end{proof}

\section{Propagation of discontinuities for the nonlocal diffusion equation}\label{propagation_of_discontinuities}
In this section, we study the propagation of discontinuities for the nonlocal diffusion equation in \eqref{eq:nonlocal_diffusion_nosource}. 
We emphasize that Theorem~\ref{thm:regularity1} implies that the nonlocal diffusion equation satisfies an  instantaneous smoothing effect when the integral  kernel is singular with $\beta>n$ and a gradual (over time) smoothing effect for when $\beta=n$. However, for  integrable kernels ($\beta<n$), the nonlocal diffusion equation is non-smoothing. In this section, we investigate this latter case further by studying the propagation of discontinuities. To this end,
given a discontinuous initial data $f\in L^2(T^n)$. Then, we  show that for certain conditions on $f$ and $\beta$, discontinuities persist and propagate. In particular, we show that in one-dimension, if $f$ is piecewise continuous, then the solution $u$ is piecewise continuous and both $f$ and $u$ share the same locations of jumps.

To study the propagation of discontinuities, we look for a decomposition of $u$, the solution of \eqref{eq:nonlocal_diffusion_nosource}, of the form
\[
u(x,t)=v(x,t)+g(t) f(x),
\]
for some function $v(x,t)$, which is continuous in $x$ and satisfies $v(x,0)=0$,   and some function $g$ satisfying   $g(0)=1$. This would imply that any discontinuity in $f$ will persist to be a discontinuity in $u$ for all $t>0$. We  show that the magnitude of a jump discontinuity decays as $t$ increases. 

We observe that $v$ satisfies the following
\begin{eqnarray*}
    v_t&=&u_t-g'(t)f(x)\\
    &=&\Ldel u - g'(t)f(x)\\
    &=&\Ldel v+g(t)\Ldel f(x) - g'(t)f(x).
\end{eqnarray*}
Since $\beta<n$, we observe that 
\[
    \Ldel f(x) = \cdel \int_{B_\delta(x)}\frac{f(y)-f(x)}{\|y-x\|^\beta}dy=h(x) - \alpha f(x),
\]
where 
\begin{equation}
\label{eq:h}
    h(x) = \cdel \int_{B_\delta(x)}\frac{f(y)}{\|y-x\|^\beta}dy=\left(\frac{\cdel}{\|\cdot\|^\beta}\chi_{B_\delta(0)}(\cdot)\right)\ast f(x),
\end{equation}
and $\alpha$ is a constant given by
\begin{equation}
\alpha = \cdel \int_{B_\delta(x)}\frac{1}{\|y-x\|^\beta}dy=\frac{2n(n+2-\beta)}{\delta^2(n-\beta)}.
\label{alpha}
\end{equation}
Therefore,
\begin{eqnarray*}
    v_t&=&\Ldel v +g(t)h(x) - \alpha g(t)f(x)-g'(t)f(x)\\
    &=&\Ldel v +g(t)h(x)-f(x)\left(\alpha g(t)-g'(t)\right).
\end{eqnarray*}
Setting $\alpha g(t)-g'(t)=0$, then $g(t) = e^{-\alpha t}$. Hence, $v$ solves
\begin{eqnarray}
    \begin{cases}
        v_t = \Ldel v +e^{-\alpha t}h(x), \;x\in T^n, t>0,\\
        v(x,0) = 0,
    \end{cases}\label{v-nonlocaldiffusionEq}
\end{eqnarray}
and therefore,
\begin{eqnarray}
    u(x,t) = v(x,t) + e^{-\alpha t}f(x).\label{v(x,t)}
\end{eqnarray}
Hence,
\begin{eqnarray}
    \hat{v}_k &=& \hat{u}_k - e^{-\alpha t}\hat{f}_k\nonumber\\
    &=& \hat{f}_k e^{m(\nu_k)t} - e^{-\alpha t}\hat{f}_k\nonumber\\
    &=& \hat{f}_k\left(1-e^{-(m(\nu_k)+\alpha)t}\right)e^{m(\nu_k)t}.
\end{eqnarray}It remains to find conditions on $f$ and $\beta$ to guarantee the continuity of $v$. Towards this end, we  make use of the following lemma, whose proof is similar to the proof of Theorem 3.2 in \cite{alali2021fourier}.
We note that the constant $\alpha$ appears in the asymptotics formula \eqref{eq:asym_multip}.

\begin{lem}\label{m+alpha}
    Let $\nu \in\Rn$ and let $\beta < n$. Suppose $\alpha$ is as defined in \eqref{alpha}. Then
    \begin{eqnarray*}
        m(\nu)+\alpha\sim \begin{cases}
            4(n+2-\beta)\Gamma\left(\frac{n}{2}+1\right)\delta^2\dfrac{\Gamma\left(\frac{n}{2}\right)\Gamma\left(\frac{n+2-\beta}{2}\right)}{\Gamma\left(\frac{\beta}{2}\right)}\left(\frac{\delta\|\nu\|}{2}\right)^{\beta-n},~\text{ if } \frac{n-1}{2}<\beta<n,\\[4ex]
            4(n+2-\beta)\Gamma\left(\frac{n}{2}+1\right)\delta^2\dfrac{(n-\beta)\Gamma\left(\frac{n}{2}\right)}{4\sqrt{\pi}}\left(\frac{\delta\|\nu\|}{2}\right)^{-\frac{n+1}{2}}, ~\text{ if } \beta\le \frac{n-1}{2}.
        \end{cases}
    \end{eqnarray*}
\end{lem}
In addition, we  make use of the following lemma.
\begin{lem}\label{v_hat_asympt}
    Let $\alpha$ be as defined in \eqref{alpha}.
    Then,
    \begin{eqnarray*}
        \left(1-e^{-(m(\nu_k)+\alpha)t}\right)e^{m(\nu_k)t}\sim
        \begin{cases}
            \dfrac{C_1te^{-\alpha t}}{\|k\|^{n-\beta}},~\text{ if } \frac{n-1}{2}<\beta<n,\\[3ex]
            \dfrac{C_2te^{-\alpha t}}{\|k\|^{\frac{n+1}{2}}},~\text{ if } \beta\le \frac{n-1}{2},
        \end{cases}
    \end{eqnarray*}
    for some  positive constants $C_1$ and $C_2$.
\end{lem}
\begin{proof}
    When $\frac{n-1}{2}<\beta<n$, then by using Lemma~\ref{m+alpha} and the definition of $\nu_k$, there exists $C_1>0$ such that
    \[
        \lim_{\|k\|\to\infty} (m(\nu_k)+\alpha)\|k\|^{n-\beta} = C_1,
    \]
    which implies that
    \[
        \frac{C_1(1-\epsilon)}{\|k\|^{n-\beta}} < m(\nu_k) + \alpha < \frac{C_1(1+\epsilon)}{\|k\|^{n-\beta}},
    \]
    for any $\epsilon>0$. Thus
    \[
        \|k\|^{n-\beta} \left(1-e^{-\frac{C_1(1-\epsilon)t}{\|k\|^{n-\beta}}}\right)<\|k\|^{n-\beta}\left(1-e^{-(m(\nu_k)+\alpha)t}\right)<\|k\|^{n-\beta} \left(1-e^{-\frac{C_1(1+\epsilon)t}{\|k\|^{n-\beta}}}\right),
    \]
    and consequently,
    \[
        C_1(1-\epsilon)t<\lim_{\|k\|\to\infty}\|k\|^{n-\beta}\left(1-e^{-(m(\nu_k)+\alpha)t}\right)<C_1(1+\epsilon)t.  
    \]
   Since $\epsilon$ is arbitrary, we obtain 
    \[
        \lim_{\|k\|\to\infty} \|k\|^{n-\beta}\left(1-e^{-(m(\nu_k)+\alpha)t}\right) = C_1 t,
    \]
    and thus,
    \[
         \left(1-e^{-(m(\nu_k)+\alpha)t}\right)e^{m(\nu_k)t} \sim \frac{C_1 t  e^{-\alpha t}}{\|k\|^{n-\beta}}.
    \]
    The proof is similar for the case $\beta \le \frac{n-1}{2}$.
\end{proof}
\begin{figure}[ht]
  \centering
  \begin{subfigure}[t]{0.24\textwidth}
    \includegraphics[width=\textwidth]{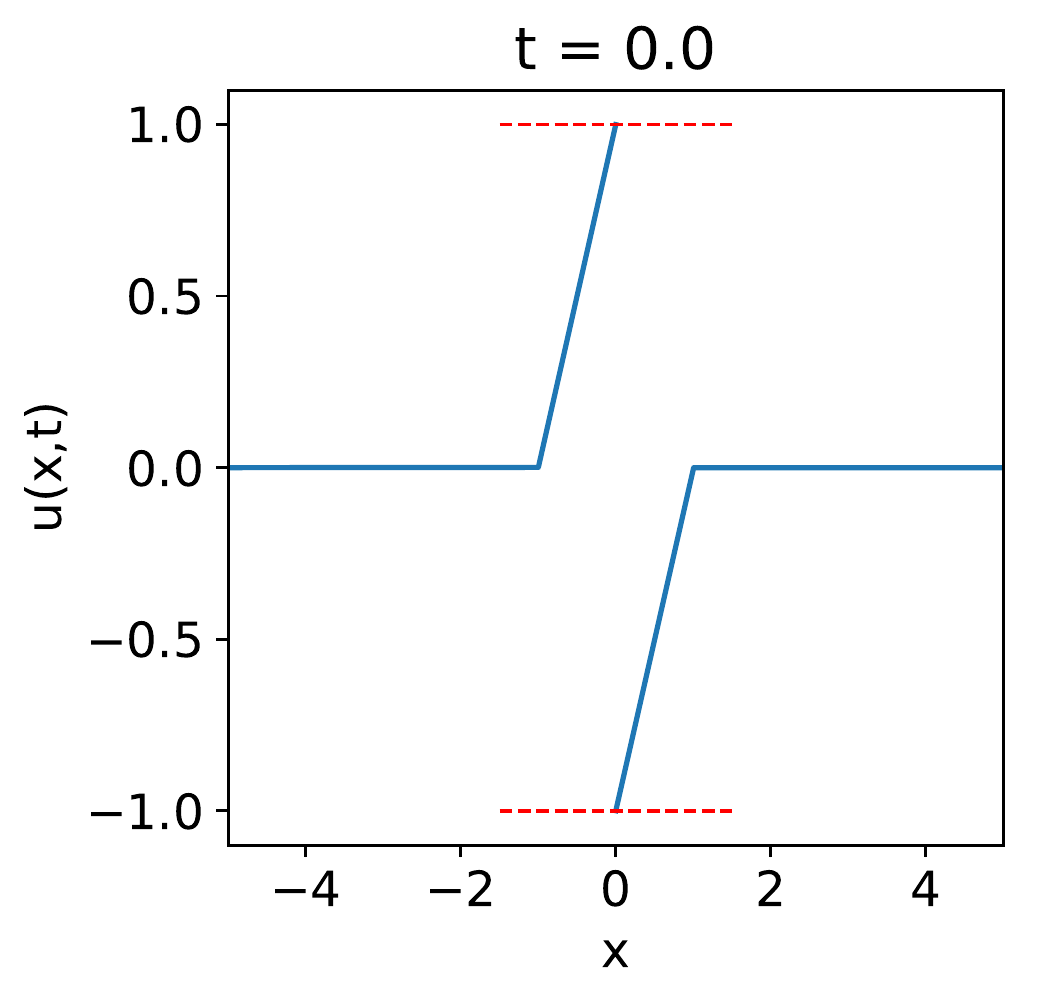}
  \end{subfigure}%
  \begin{subfigure}[t]{0.24\textwidth}
    \includegraphics[width=\textwidth]{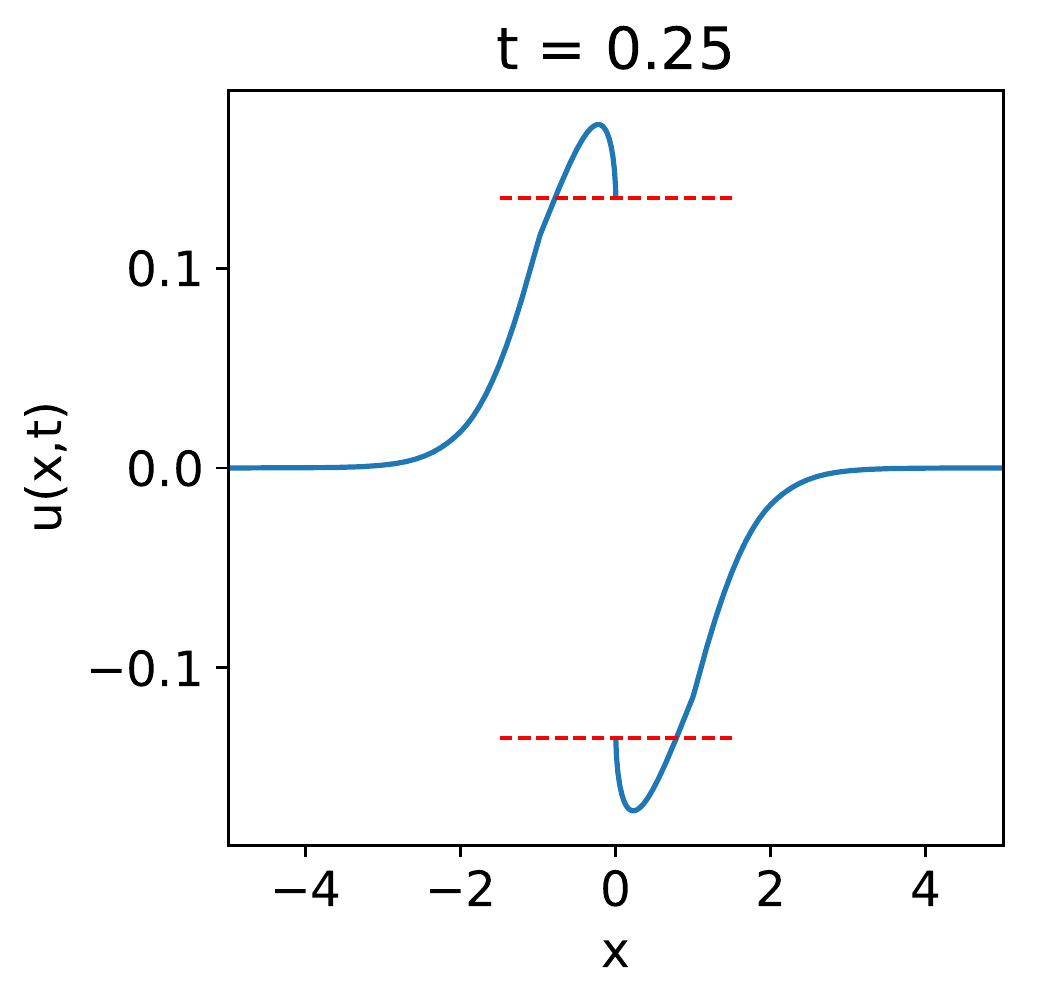}
  \end{subfigure}%
  \begin{subfigure}[t]{0.24\textwidth}
    \includegraphics[width=\textwidth]{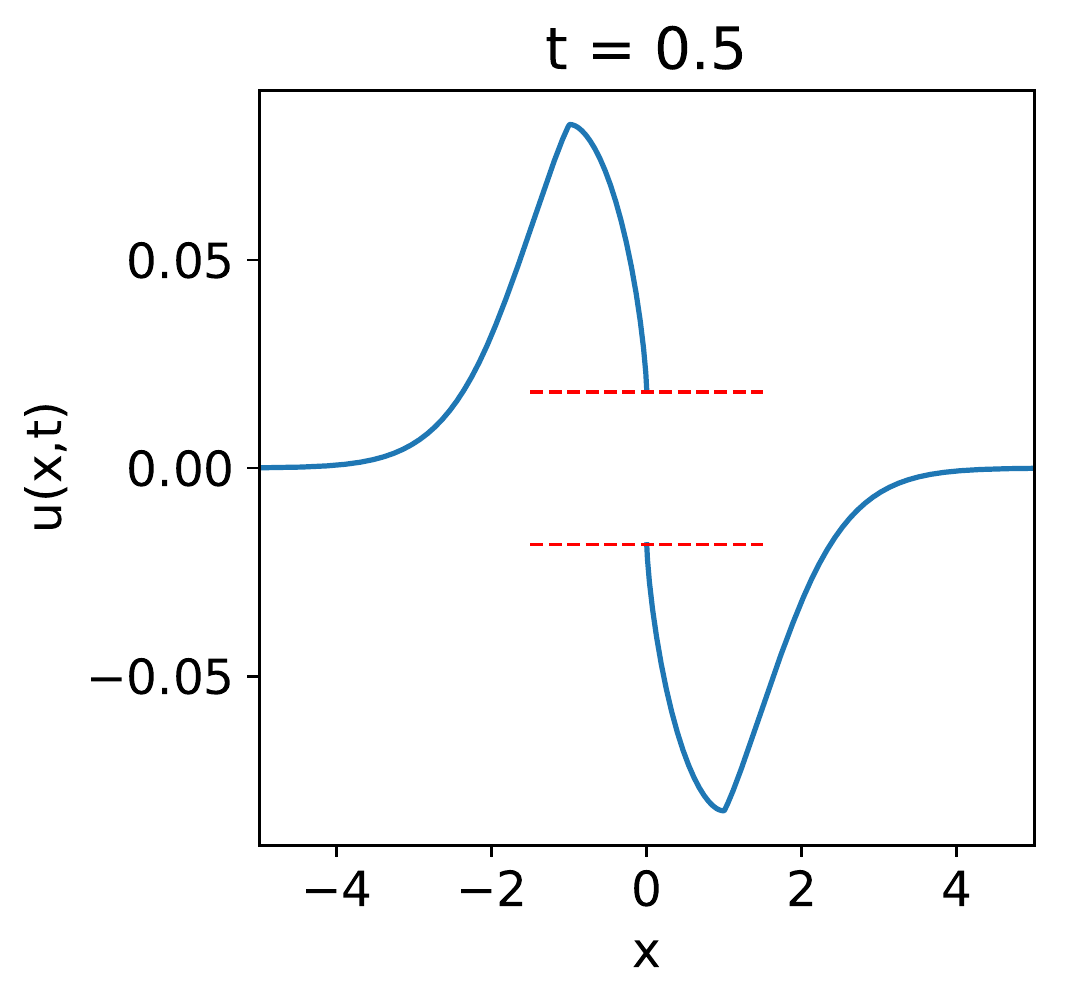}
  \end{subfigure}%
  \begin{subfigure}[t]{0.24\textwidth}
    \includegraphics[width=\textwidth]{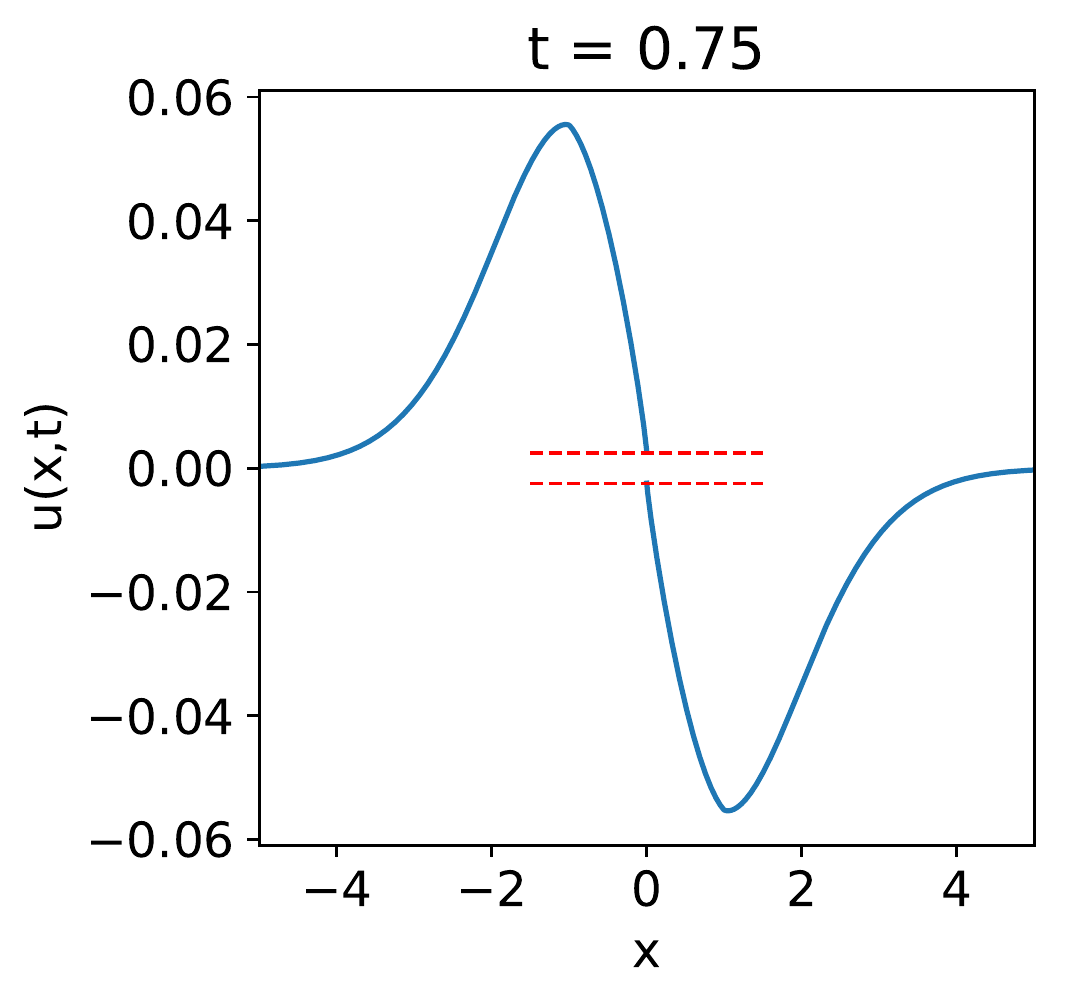}
  \end{subfigure}
  \caption{The slow decay of a discontinuity in the nonlocal diffusion
    equation with $\beta<n$.} 
  \label{fig:jump-decay}
\end{figure}

Conditions on $f$ and $\beta$ to guarantee the continuity of $v$ are given in the following result.
\begin{theorem}\label{continuity_of_v}
   Let $v(x,t)$ be as in  \eqref{v(x,t)} and assume that $\hat{f}_k$ satisfies
   \[
        \hat{f}_k \sim
        \begin{cases}
            \frac{1}{\|k\|^{\beta+\epsilon}}, \text{ if } \frac{n-1}{2}<\beta<n,\\
            \frac{1}{\|k\|^{\frac{n-1}{2}+\zeta}}, \text{ if } \beta\le\frac{n-1}{2},
        \end{cases}
   \]
    for $\epsilon, \zeta>0$. Then, $v(x,t)$ is continuous.
\end{theorem}
\begin{proof}
    For $\frac{n-1}{2} < \beta < n$ with $\hat{f}_k\sim \frac{1}{\|k\|^{\beta+\epsilon}}$, then by using Lemma~\ref{v_hat_asympt}, we have
    \begin{eqnarray*}
        \hat{v}_k &=& \hat{f}_k\left(1-e^{-(m(\nu_k)+\alpha)t}\right)e^{m(\nu_k)t}\sim \dfrac{Cte^{-\alpha t}}{\|k\|^{n+\epsilon}}.
    \end{eqnarray*}
    Similarly, for $\beta \le \frac{n-1}{2}$ with $\hat{f}_k\sim \frac{1}{\|k\|^{\frac{n-1}{2}+\zeta}}$, then
    \begin{eqnarray*}
        \hat{v}_k &=& \hat{f}_k\left(1-e^{-(m(\nu_k)+\alpha)t}\right)e^{m(\nu_k)t}
        \sim \dfrac{Cte^{-\alpha t}}{\|k\|^{n+\zeta}}.
    \end{eqnarray*}
    By Proposition~3.3.12 in \cite{grafakos2008classical}, we conclude that for $t>0$, $v(\cdot,t)$ is continuous in both cases.
\end{proof}
The following theorem summarizes the results in this section.
\begin{theorem}\label{u_f_discontinuity}
    Let $\beta<n$ and let $u$ be as given in \eqref{v(x,t)}
    and assume that
    \[
        \hat{f}_k \sim
        \begin{cases}
            \frac{1}{\|k\|^{\beta+\epsilon}}, \text{ if } \frac{n-1}{2}<\beta<n,\\
            \frac{1}{\|k\|^{\frac{n-1}{2}+\zeta}}, \text{ if } \beta\le\frac{n-1}{2},
        \end{cases}
   \]
    for some $\epsilon, \zeta>0$.  Then, if $f$ is discontinuous at $x$ then $u$ is discontinuous at $x$.
\end{theorem}
\begin{cor}
    If 
    $f\in L^2(T)$ is piecewise continuous, then $u$ is piecewise continuous and $f$ and $u$ share the same locations of jumps. Furthermore, the magnitude of a jump decays as $t$ increases.
\end{cor}
This is an immediate consequence of Theorem~\ref{u_f_discontinuity}, since for a piecewise continuous function $f\in L^2(T)$, we have $\hat{f}_k\sim\dfrac{C}{|k|}$, for some $C>0$.

A one-dimensional example for the propagation of a discontinuity in the nonlocal diffusion equation is described below.  
Figure~\ref{fig:jump-decay} shows the results of a numerical solution
to the periodic nonlocal diffusion problem $u_t=\Ldel u$ on the interval $(-10,10)$
with $\delta = 1$, $\beta = 1/3$, and initial condition
\begin{equation*}
  u(x,0) =
  \begin{cases}
    x+1 & \text{if }-1<x\le 0,\\
    x-1 & \text{if }0<x<1,\\
    0 & \text{otherwise}.
  \end{cases}
\end{equation*}

In Figure~\ref{fig:jump-decay}, function values for $x<0$ were plotted separately from those for $x>0$ so that the jump is apparent.  The dashed lines indicate the values $\pm e^{-\alpha t}$, showing the theoretical extremes of the jump.

\clearpage
\bibliographystyle{acm}
\bibliography{refs}

\end{document}